\documentclass[12pt]{amsart}
\usepackage{amsmath,amssymb,amsthm,amscd}
\newtheorem{thm}{Theorem}
\newtheorem*{thmnonum}{Theorem}
\newtheorem{lem}[thm]{Lemma}

\newtheorem{prop}[thm]{Proposition}
\newtheorem{cor}[thm]{Corollary}
\newtheorem*{rem}{Remark}
\newtheorem*{ack}{Acknowledgements}
\newtheorem*{ex}{Example}
\newtheorem{conjec}[thm]{Conjecture}

\newcommand{\SL}{{\rm SL}}
\newcommand{\GL}{{\rm GL}}
\newcommand{\C}{\mathbb{C}}
\renewcommand{\H}{\mathbb{H}}
\newcommand{\Q}{\mathbb{Q}}
\newcommand{\Z}{\mathbb{Z}}
\newcommand{\R}{\mathbb{R}}
\newcommand{\A}{\mathbb{A}}
\newcommand{\F}{\mathbb{F}}
\newcommand{\Aut}{{\rm Aut}}

\newcommand{\legen}[2]{\genfrac{(}{)}{}{}{#1}{#2}}
\newcommand{\Gen}{{\rm Gen}}
\newcommand{\cond}{{\rm cond}}
\newcommand{\Ad}{{\rm Ad}}
\newcommand{\Sym}{{\rm Sym}}
\newcommand{\Gal}{{\rm Gal}}
\newcommand{\Res}{{\rm Res}}
\newcommand{\RE}[1]{{\rm Re}\left(#1\right)}
\newcommand{\ord}{{\rm ord}}
\newcommand{\Tr}{{\rm Tr}}
\begin{document}

\title[Quadratic forms representing all odd integers]{Quadratic forms representing all odd positive integers}
\author{Jeremy Rouse}
\address{Department of Mathematics, Wake Forest University,
  Winston-Salem, NC 27109}
\email{rouseja@wfu.edu}
\thanks{The author was supported by NSF grant DMS-0901090}
\begin{abstract}
We consider the problem of classifying all positive-definite integer-valued
quadratic forms that represent all positive odd integers. Kaplansky considered
this problem for ternary forms, giving a list of 23 candidates, and proving
that 19 of those represent all positive odds. (Jagy later dealt with
a 20th candidate.) Assuming that the remaining three forms represent
all positive odds, we prove that an arbitrary, positive-definite
quadratic form represents all positive odds if and only if it
represents the odd numbers from 1 up to 451. This result is analogous to
Bhargava and Hanke's celebrated 290-theorem. In addition, we prove that
these three remaining ternaries represent all positive odd integers, assuming
the Generalized Riemann Hypothesis.

This result is made possible by a new analytic method for bounding the cusp
constants of integer-valued quaternary quadratic forms $Q$ with fundamental
discriminant. This method is based on the analytic properties of Rankin-Selberg
$L$-functions, and we use it to prove that if $Q$ is a quaternary form with
fundamental discriminant, the largest locally represented
integer $n$ for which $Q(\vec{x}) = n$ has no integer solutions is
$O(D^{2 + \epsilon})$.
\end{abstract}

\maketitle

\section{Introduction and Statement of Results}
\label{intro}

The study of which integers are represented by a given quadratic
form is an old one. In 1640, Fermat stated his
conjecture that every prime number $p \equiv 1 \pmod{4}$
can be written in the form $x^{2} + y^{2}$. In the
next century, Euler proved Fermat's conjecture and worked seriously on
related problems and generalizations. In 1770, Lagrange proved that
every positive integer is a sum of four squares. In 1798, Legendre classified
the integers that could be represented as a sum of three squares. This result
is deeper and more difficult than either of the two-square or four-square
theorems.

Motivated by Lagrange's result, it is natural to ask about the collection
of quadratic forms that represent all positive integers, or more generally to
fix in advance a collection $S$ of integers, and ask about quadratic forms
that represent all numbers in $S$. The first result in this direction is
due to Ramanujan \cite{Ramanujan}, who in 1916 gave a list of 55 quadratic
forms of the form
\[
  Q(x,y,z,w) = ax^{2} + by^{2} + cz^{2} + dw^{2},
\]
and asserted that this list consisted precisely of the forms
(of this prescribed shape) that represent all positive integers. Dickson
\cite{Dickson} confirmed Ramanujan's statement (modulo the error that the form
$x^{2} + 2y^{2} + 5z^{2} + 5w^{2}$ was included on Ramanujan's list and represents
every positive integer except $15$), and coined the term \emph{universal} to
describe quadratic forms that represent all positive integers.

A positive-definite quadratic form $Q$ is called \emph{integer-matrix}
if it can be written in the form
\[
  Q(\vec{x}) = \vec{x}^{T} M \vec{x}
\]
where the entries of $M$ are integers. This is equivalent to saying that
if
\[
  Q(\vec{x}) = \sum_{i=1}^{n} \sum_{j \geq i}^{n} a_{ij} x_{i} x_{j},
\]
then $a_{ij}$ is even if $i \ne j$. A form $Q$ is called \emph{integer-valued}
if the cross-terms $a_{ij}$ are allowed to be odd. In 1948,
Willerding \cite{Willerding} classified universal integer-matrix quaternary
forms, giving a list of $178$ such forms.

The following result classifying integer-matrix universal forms
(in any number of variables) was proven by Conway and Schneeberger in 1993
(see \cite{SchneebergerThesis}).
\begin{thmnonum}[``The 15-Theorem'']
\label{CS}
A positive-definite integer-matrix quadratic form is universal if and only if
it represents the numbers
\[
  1, 2, 3, 5, 6, 7, 10, 14, \text{ and } 15.
\]
\end{thmnonum}

This theorem was elegantly reproven by Bhargava in 2000 (see
\cite{Bhar}).  Bhargava's approach is to work with integral lattices,
and to classify escalator lattices - lattices that must be inside any
lattice whose corresponding quadratic form represents all positive
integers. As a consequence, Bhargava found that there are in fact
$204$ universal quaternary integer-matrix forms. Willerding had missed
36 universal forms, listed one universal form twice, and listed nine
forms which were not universal.

Bhargava's approach is quite general. Indeed, he has proven that
for any infinite set $S$, there is a finite subset $S_{0}$ of $S$
so that any positive-definite integral quadratic form represents all numbers
in $S$ if it represents the numbers in $S_{0}$. Here the notion of
integral quadratic form can mean either integer-matrix or integer-valued
(and the set $S_{0}$ depends on which notion is used).
Bhargava proves that if $S$ is the set of odd numbers, then
any integer-matrix form represents everything in $S$ if it represents
everything in $S_{0} = \{ 1, 3, 5, 7, 11, 15, 33 \}$. He also determines
$S_{0}$ in the case that $S$ is the set of prime numbers
(again for integer-matrix forms); the largest element of $S_{0}$ is $73$.
(These results are stated in \cite{MKim}.)

While working on the 15-Theorem, Conway and Schneeberger were led to conjecture
that every integer-valued quadratic form that represents the
positive integers between 1 and 290 must be universal. Bhargava and
Hanke's celebrated 290-Theorem proves this conjecture (see \cite{BH}).
Their result is the following.
\begin{thmnonum}[``The 290-Theorem'']
If a positive-definite integer-valued quadratic form
represents the twenty-nine integers
\begin{align*}
 & 1, 2, 3, 5, 6, 7, 10, 13, 14, 15, 17, 19, 21, 22, 23, 26, 29\\
 & 30, 31, 34, 35, 37, 42, 58, 93, 110, 145, 203, \text{and } 290,
\end{align*}
then it represents all positive integers.
\end{thmnonum}
They also show that every one of the twenty-nine integers above is necessary.
Indeed, for every integer $t$ on this list, there is a positive-definite
integer-valued quadratic form that represents every positive integer except $t$.
As a consequence of the 290-Theorem, they are able to prove that there are
exactly $6436$ universal integer-valued quaternary quadratic forms.

A \emph{regular} positive-definite quadratic form is a form
$Q(\vec{x})$ with the property that if $n$ is a positive integer and
$Q(\vec{x}) = n$ is solvable in $\Z_{p}$ for all primes $p$, then
$Q(\vec{x}) = n$ is solvable in $\Z$. Willerding and Bhargava make use of
regular forms in their work on universal integer-matrix forms.

In the work of Bhargava and Hanke, they switch to using the analytic theory of
modular forms, as they need to completely understand more than $6000$
quaternary quadratic forms to prove the 290-Theorem. This technique is
very general and requires extensive computer computations.

In this paper, we will consider the problem of determining a finite
set $S_{0}$ with the property that a positive-definite integer-valued
quadratic form represents every odd positive integer if and only if it
represents everything in $S_{0}$. One difference between this
problem and the case when $S$ is all positive integers is that there are
ternary quadratic forms that represent all odd integers, and it is
necessary to classify these. In \cite{Kap}, Kaplansky considers this
problem. He proves that there are at most $23$ such forms, and gives
proofs that $19$ of the $23$ represent all odd positive integers. He
describes the remaining four as ``plausible candidates'' and indicates
that they represent every odd positive integer less than $2^{14}$. In
\cite{Jagy}, Jagy proved that one of Kaplansky's candidates, $x^{2} +
3y^{2} + 11z^{2} + xy + 7yz$, represents all positive odds. The
remaining three have yet to be treated.

\begin{conjec}
\label{Kaplansky}
Each of the ternary quadratic forms
\begin{align*}
  & x^{2} + 2y^{2} + 5z^{2} + xz\\
  & x^{2} + 3y^{2} + 6z^{2} + xy + 2yz\\
  & x^{2} + 3y^{2} + 7z^{2} + xy + xz
\end{align*}
represents all positive odd integers.
\end{conjec}

\begin{rem}
There is (at present) no general algorithm for determining the integers 
represented by a positive-definite ternary quadratic form $Q$. If $n$ is
a large positive integer, the number of representations by $Q$ is closely 
approximated by an expression involving the class number of an imaginary 
quadratic field (depending on $n$, see Section~\ref{kappf} for more detail).
Bounds for class numbers are closely tied to the question of whether 
a quadratic Dirichlet $L$-function can have a Siegel zero, and this is one
of the most notorious unsolved problems in number theory.
\end{rem} 

We can now state our first main result.

\begin{thm}[``The 451-Theorem'']
\label{451}
Assume Conjecture~\ref{Kaplansky}. Then, a positive-definite, integer-valued
quadratic form represents all positive odd integers if and only if it
represents the 46 integers
\begin{align*}
& 1, 3, 5, 7, 11, 13, 15, 17, 19, 21, 23, 29, 31, 33, 35, 37, 39, 41,
47,\\
& 51, 53, 57, 59, 77, 83, 85, 87, 89, 91, 93, 105, 119, 123, 133, 137,\\
& 143, 145, 187, 195, 203, 205, 209, 231, 319, 385, \text{and } 451.
\end{align*}
\end{thm}

As was the case for the 290-Theorem, all of the integers above are necessary.
\begin{cor}
\label{crit}
For every one of the $46$ integers $t$ on the list above, there is a
positive-definite, integer-valued quadratic form that represents every odd
number except $t$.
\end{cor}

We also have an analogue of results proven in \cite{Bhar} and \cite{BH}
regarding what happens if the largest number is omitted.
\begin{cor}
\label{smallbig}
Assume Conjecture~\ref{Kaplansky}. If a positive-definite, integer-valued
quadratic form represents every positive odd number less than $451$, it
represents every odd number greater than $451$.
\end{cor}

As a consequence of the 451-Theorem, we can classify integer-valued
quaternary forms that represent all positive odd integers.
\begin{cor}
\label{oddclass} Assume Conjecture~\ref{Kaplansky}.
Suppose that $Q$ is a positive-definite, integer-valued, quaternary
quadratic form that represents all positive odds. Then either:

(a) $Q$ represents one of the $23$ ternary quadratic forms which
represents all positive odds, or

(b) $Q$ is one of $21756$ quaternary forms.
\end{cor}

To prove the 451-Theorem, we must determine the positive, odd, squarefree
integers represented by $24888$ quaternary quadratic forms $Q$. Any form
that represents all positive odd numbers must represent either one
of Kaplansky's ternaries, or one of these $24888$ quaternary forms.
This makes the analysis of forms in five or more variables much simpler.

To analyze the quaternary forms, we use a combination of four
methods. The first method checks to see if a given quaternary
represents any of the $23$ ternaries listed by Kaplansky. If so, it
must represent all positive odds (assuming
Conjecture~\ref{Kaplansky}).

The second method attempts to find, given the integer lattice $L$
corresponding to $Q$, a ternary sublattice $L'$ so that the
quadratic form corresponding to $L'$ is regular, and the lattice
$L' \oplus \left(L'\right)^{\perp}$ locally represents all positive odds.
We make use of the classification of regular ternary quadratic
forms due to Jagy, Kaplansky, and Schiemann \cite{JKS}. This is a version of
the technique used by Willerding and Bhargava.

The last two methods are analytic in nature. If $Q$ is a positive-definite,
integer-valued quaternary quadratic form, then
\[
  \theta_{Q}(z) = \sum_{n=0}^{\infty} r_{Q}(n) q^{n} \in
  M_{2}(\Gamma_{0}(N), \chi), \quad q = e^{2 \pi i z}
\]
is a modular form of weight $2$. We can decompose $\theta_{Q}(z)$ as
\begin{align*}
  \theta_{Q}(z) &= E(z) + C(z)\\
  &= \sum_{n=0}^{\infty} a_{E}(n) q^{n} + \sum_{n=1}^{\infty} a_{C}(n) q^{n}.
\end{align*}
Theorem 5.7 of \cite{Hanke} gives the lower bound
\[
  a_{E}(n) \geq C_{E} n \prod_{\substack{p | n \\ \chi(p) = -1}} \frac{p-1}{p+1}
\]
for some some constant $C_{E}$, depending on $Q$, provided $n$ is squarefree
and locally represented by $Q$. We may decompose the form $C(z)$ into
a linear combination of newforms (and the images of newforms under $V(d)$).
It is known that the $n$th Fourier coefficient of a newform
of weight $2$ is bounded by $d(n) n^{1/2}$ (first proven by Eichler,
Shimura, and Igusa in the weight $2$ case, and Deligne in the general case),
and so there is a constant $C_{Q}$ so that
\[
  |a_{C}(n)| \leq C_{Q} d(n) n^{1/2}.
\]
If we can compute or bound the constants $C_{E}$ and $C_{Q}$, we can determine
the squarefree integers represented by $Q$ via a finite computation.

One method we use is to compute the constant $C_{Q}$ explicitly, by computing
the Fourier expansions of all newforms and expressing $C(z)$ in terms of them.
This method is the approach taken by Bhargava and Hanke to all of
the cases they consider in \cite{BH}, and works very well when the coefficient
fields of the newforms are reasonably small.

However, in Bhargava and Hanke's cases, the newforms in the decomposition have
coefficients in number fields of degree as high as 672. Jonathan Hanke reports
that computations of $C_{Q}$ take weeks of CPU time on current hardware. In our case, we
must consider spaces that have Galois conjugacy classes of newforms of
size at least 1312, and for degrees as large as this, this explicit,
direct approach is impossible from a practical standpoint.

These large degree number fields only arise in cases when
$S_{2}(\Gamma_{0}(N), \chi)$ is close to being irreducible as a Hecke
module. If the conductor of $\chi$ is not primitive, we have a
decomposition of $S_{2}(\Gamma_{0}(N), \chi)$ into old and new
subspaces which are Hecke stable. For this reason, we develop a new method to
bound the constant $C_{Q}$ without explicitly computing the newform
decomposition of $C$ which applies when the discriminant of the
quadratic form $Q$ is a fundamental discriminant.

Our method allows us to improve significantly the bounds given
in the literature on the largest integer $n$ that is not represented
by a form $Q$ satisfying appropriate local conditions. For a form
$Q(\vec{x}) = \frac{1}{2} \vec{x}^{T} A \vec{x}$,
where $A$ has integer entries and even diagonal entries,
let $D(Q) = \det A$ be the discriminant of $Q$, and let $N(Q)$
be the level of $Q$. In \cite{SP}, Schulze-Pillot proves the following
result.
\begin{thmnonum}
Suppose that $Q$ is a positive-definite, integer-valued, quaternary
quadratic form with level $N(Q)$. If $n$ is a positive integer so that
$Q(\vec{x}) = n$ has primitive solutions in $\Z_{p}$ for all primes $p$,
and
\[
  n \gg N(Q)^{14+\epsilon},
\]
then $n$ is represented by $Q$.
\end{thmnonum}

\begin{rem}
We have given a simplified version of Schulze-Pillot's result. The bound
Schulze-Pillot gives is completely explicit.
\end{rem}

In \cite{BrowningDietmann}, Browning and Dietmann use the circle method
to study integer-matrix quadratic forms $Q(\vec{x}) = \vec{x}^{T} A \vec{x}$.
A pair $(Q, k)$ (consisting of a quaternary quadratic form and a positive integer $k$)
satisfies the strong local solubility condition if for all primes $p$ there is a
vector $\vec{x} \in \Z^{4}$ so that
\[
  Q(\vec{x}) \equiv k \pmod{p^{1 + 2 \tau_{p}}}
\]
and $p \nmid A\vec{x}$. Here $\tau_{p}$ is zero if $p$ is odd and is one if
$p = 2$. Their result about quaternary forms is the following.

\begin{thmnonum}
Assume the notation above and let $\|Q\|$ denote
the largest entry in the Gram matrix $A$ of $Q$. Let $\mathfrak{k}_{4}^{*}(Q)$
be the largest positive integer $k$ that satisfies the strong local solubility
condition but is not represented by $Q$. Then
\[
  \mathfrak{k}_{4}^{*}(Q) \ll D(Q)^{2} \|Q\|^{8+\epsilon}.
\]
\end{thmnonum}

\begin{rem}
Depending on the quaternary form $Q$, the bound above is between
$D(Q)^{4 + \epsilon}$ and $D(Q)^{10 + \epsilon}$. For a ``generic'' quaternary form
with small coefficients, we have $\|Q\| \ll D(Q)^{1/4}$ and the bound
$D(Q)^{4 + \epsilon}$.
\end{rem}

Our next main result is a significant improvement on the result of
Browning and Dietmann in the two cases that $D(Q)$ is a fundamental
discriminant, or that $N(Q)$ is a fundamental discriminant
and $D(Q) = N(Q)^{3}$.

\begin{thm}
\label{bound}
Suppose that $Q$ is a positive-definite integer-valued quaternary
quadratic form with fundamental discriminant $D(Q)$. If $n$ is locally
represented by $Q$, but is not represented by $Q$, then
\[
  n \ll D(Q)^{2 + \epsilon}.
\]
If $Q$ is a form whose level $N(Q)$ is a fundamental discriminant and
$D(Q) = N(Q)^{3}$ and $n$ is locally represented by $Q$ but not
represented, then
\[
  n \ll D(Q)^{1 + \epsilon}.
\]
\end{thm}

\begin{rem}
  To compare our result with that of Browning and Dietmann we need
  some bound on $\|Q\|$. The best general bounds we can give in the two cases
are $\|Q\| \ll D(Q)$ and $\|Q\| \ll
  D(Q)^{1/3}$, respectively. Their result then yields the bounds $n
  \ll D(Q)^{10+\epsilon}$ and $n \ll D(Q)^{14/3 + \epsilon}$,
  respectively.
\end{rem}

\begin{rem}
In Theorem 6.3 of \cite{Hanke}, bounds on the largest non-represented
integer that is locally represented and has a priori bounded 
divisibility by the anisotropic primes in terms of the constant $C_{Q}$.
Our contribution is to give a strong bound on $C_{Q}$ as a function of
$D(Q)$ (in the case that $D(Q)$ is a fundamental discriminant).
\end{rem}

\begin{rem}
Theorem~\ref{451} and Theorem~\ref{bound} both rely on a formula for the 
Petersson norm of the cusp form $C(z)$. This can be translated into
a bound on the cusp constant $C_{Q}$ provided lower bounds on the Petersson
norms of the newform constituents of $C(z)$ are available. Theorem~\ref{bound}
is ineffective because of the possibility of a Siegel
zero of arising from an $L$-function of a CM newform $g$. However,
for a given $Q$, all such $g$ can be enumerated and the relevant $L$-values
computed numerically. This allows one to extract an explicit bound for the 
cusp constant $C_{Q}$. 
\end{rem}

Our method is similar to the approach of Schulze-Pillot \cite{SP} and
Fomenko (see \cite{Fom} and \cite{Fom2}). We obtain upper bounds on
$\langle C, C \rangle$ and lower bounds on $\langle g_{i}, g_{i}
\rangle$ using the theory of Rankin-Selberg $L$-functions.

If $g_{i} = \sum_{n=1}^{\infty} a(n) q^{n}$ and $g_{j} = \sum_{n=1}^{\infty}
b(n) q^{n}$ are two newforms in $S_{k}(\Gamma_{0}(N), \chi)$ with
\begin{align*}
  L(g_{i}, s) &:= \sum_{n=1}^{\infty} \frac{a(n)}{n^{s + \frac{k-1}{2}}}
  = \prod_{p} (1 - \alpha_{p} p^{-s})^{-1}
  (1 - \beta_{p} p^{-s})^{-1},\\
  L(g_{j}, s) &:= \sum_{n=1}^{\infty} \frac{b(n)}{n^{s + \frac{k-1}{2}}}
  = \prod_{p} (1 - \gamma_{p} p^{-s})^{-1} (1 - \delta_{p} p^{-s})^{-1},
\end{align*}
the Rankin-Selberg convolution $L$-function of $g_{i}$ and $g_{j}$ is
\[
  L(g_{i} \otimes g_{j}, s) = \prod_{p | N} L_{p}(g_{i} \otimes g_{j}, s)
  \prod_{p \nmid N}
  (1 - \alpha_{p} \gamma_{p} p^{-s})^{-1} (1 - \alpha_{p} \delta_{p} p^{-s})^{-1}
  (1 - \beta_{p} \gamma_{p} p^{-s})^{-1} (1 - \beta_{p} \delta_{p} p^{-s})^{-1}.
\]
Here $L_{p}(g_{i} \otimes g_{j}, s)$ is an appropriate local factor
predicted by the local Langlands correspondence (and worked out explicitly
by Li in \cite{Li2}). If $g_{j} =
\overline{g_{i}}$, then $L(g_{i} \otimes g_{j}, s)$ has a pole at $s =
1$ with residue equal to an explicit factor times the Petersson norm
of $g_{i}$. If the factor $L_{p}(g_{i} \otimes g_{j}, s)$ is chosen
appropriately, then $L(g_{i} \otimes g_{j}, s)$ will have a
meromorphic continuation to all of $\C$ (with the only possible pole
occurring when $s = 1$ and $g_{j} = \overline{g_{i}}$) and a functional
equation of the usual type.

In the appendix to \cite{GHL}, Goldfeld, Hoffstein, and Lieman show
that $L(g_{i} \otimes \overline{g_{i}}, s)$ has no Siegel zero. We
make effective the result of Goldfeld, Hoffstein and Lieman, and
translate this into an explicit lower bound for $\langle g_{i}, g_{i}
\rangle$.

To give a bound on the Petersson norm of $C$, we need to extend our
theory of Rankin-Selberg $L$-functions to arbitrary elements of
$S_{2}(\Gamma_{0}(N), \chi)$. If $f, g \in S_{2}(\Gamma_{0}(N), \chi)$
we decompose
\begin{align*}
  f = \sum_{i=1}^{u} c_{i} g_{i}, \text{ and }
  g = \sum_{j=1}^{u} d_{j} g_{j}
\end{align*}
into linear combinations of newforms and define
\[
  L(f \otimes g, s) = \sum_{i=1}^{u} \sum_{j=1}^{u} c_{i} d_{j}
L(g_{i} \otimes g_{j}, s).
\]
However, the prediction for $L_{p}(g_{i} \otimes g_{j}, s)$ that
comes from the local Langlands correspondence makes it so the formula
that takes a pair $(f,g)$ and expresses $L(f \otimes g, s)$ in terms
of the Fourier coefficients of $f$ and $g$ is not, in general, bilinear.
For this reason, there is no straightforward way to use these Rankin-Selberg
$L$-functions to compute $\langle C, C \rangle$.

However, we prove that bilinearity holds if when restricted to
\[
  S_{2}^{-}(\Gamma_{0}(N), \chi)
  = \left\{ \sum_{n=1}^{\infty} a(n) q^{n} \in S_{2}(\Gamma_{0}(N), \chi) :
  a(n) = 0 \text{ if } \chi(n) = 1 \right\}.
\]
Hence, for forms $f \in S_{2}^{-}(\Gamma_{0}(N), \chi)$, $L(f \otimes
\overline{f}, s)$ has an analytic continuation, functional equation, relation
between $\Res_{s=1} L(f \otimes \overline{f}, s)$ and the Petersson norm of $f$,
and a Dirichlet series representation that can be expressed in terms
of the coefficients of $f$. For an arbitrary quadratic form
$Q$, the cuspidal part of its theta function $C$ need not be in
$S_{2}^{-}(\Gamma_{0}(N), \chi)$.  The assumption that $Q =
\frac{1}{2} \vec{x}^{T} A \vec{x}$ where $D(Q) = \det(A)$ is a fundamental
discriminant implies that if $Q^{*} = \frac{1}{2} \vec{x}^{T} N A^{-1}
\vec{x}$, then $\theta_{Q^{*}} = E^{*} + C^{*}$ and $C^{*} \in
S_{2}^{-}(\Gamma_{0}(N), \chi)$. Also, $\langle C^{*}, C^{*} \rangle
= \frac{1}{\sqrt{N}} \langle C, C \rangle$. Using the functional
equation for $L(C^{*} \otimes C^{*}, s)$, we are able to derive
a formula for $\langle C^{*}, C^{*} \rangle$ (see Proposition~\ref{petform}).
This formula is useful both theoretically (in the proof of Theorem~\ref{bound})
and practically. As an added bonus, the Fourier coefficients of $C^{*}$
are faster to compute than those of $C$, since the discriminant of the
form $Q^{*}$ is much larger than that of $Q$.

The method described above gives a much faster algorithm for
determining the integers represented by a quadratic form $Q$ with
fundamental discriminant.  In particular, we can determine the odd
squarefree integers represented by a quadratic form $Q$ with
$\theta_{Q} \in M_{2}(\Gamma_{0}(6780), \chi_{6780})$ using 26 minutes
of CPU time (see Example~\ref{fundisc} of Section~\ref{pfof451}). This
and subsequent CPU time estimates refer to computations done by the
author on a 3.2GHz Intel Xeon W3565 processor.

Finally, we return to Conjecture~\ref{Kaplansky}. For a ternary quadratic
form $Q$, the analytic theory gives a formula of the type
\[
  r_{Q}(n) = a h(-bn) + B(n)
\]
provided $n$ is squarefree and locally represented by $Q$. Here,
$h(-bn)$ is the class number of $\Q(\sqrt{-bn})$, and $B(n)$ is
the $n$th coefficient of a weight $3/2$ cusp form, and the
constants $a$, $b$, and the form of $B(n)$ depend on the
image of $n$ in $\Q_{p}^{\times}/(\Q_{p}^{\times})^{2}$ for primes $p$ dividing
the level of $Q$.

Given the ineffective bound $h(-bn) \gg n^{1/2 - \epsilon}$ for all $\epsilon > 0$,
a bound of the shape $|B(n)| \ll n^{1/2 - \delta}$ for some fixed $\delta > 0$
is necessary to show that $r_{Q}(n) > 0$ for large $n$. Waldspurger's
theorem relates $B(n)$ to the central $L$-values of quadratic twists of
a fixed number of weight $2$ modular forms, and so a non-trivial bound
on $B(n)$ is equivalent to a sub-convexity estimate for these central
$L$-values. Estimates of this type were given by Parson \cite{Parson}
for coefficients of half-integer weight forms of weight $\geq 5/2$
and improved by Iwaniec \cite{FCMHI}. Duke's result in \cite{Duke}
handles the weight $3/2$ case and gives a bound with $\delta = 1/28$.
Bykovskii (see \cite{Bykovskii}) gave a bound with $\delta = 1/16$
valid for weights greater than or equal to $5/2$, and Blomer
and Harcos \cite{BlomerHarcos} obtain $\delta = 1/16$ for weight $3/2$.

Given that the bound on the class number is ineffective, we follow the
conditional approach pioneered by Ono and Soundararajan
\cite{OnoSound}, Kane \cite{Kane}, and simplified by Chandee
\cite{Chandee}.

\begin{thm}
\label{GRHimplies} The Generalized Riemann Hypothesis implies
Conjecture~\ref{Kaplansky}.
\end{thm}

An outline of the paper is as follows. In Section~\ref{back}
we will review necessary background about quadratic forms and modular forms.
In Section~\ref{RS} we develop the theory of Rankin-Selberg $L$-functions
which we will use in Section~\ref{pfofbound} to prove Theorem~\ref{bound}.
In Section~\ref{pfof451} we will prove the 451-Theorem, and in
Section~\ref{kappf} we will prove Theorem~\ref{GRHimplies}.

\begin{ack}
The author used the computer software package Magma \cite{Magma} version 2.17-10
extensively for the computations that prove the 451-Theorem. The author would
also like to thank Manjul Bhargava, Jonathan Hanke, David Hansen, Ben Kane,
and Ken Ono for helpful conversations. This work was completed over the course
of five years at the University of Wisconsin-Madison, the University
of Illinois at Urbana-Champaign, and Wake Forest University. The author
wishes to thank each of these institutions for their support of this work.
Finally, the author wishes to acknowledge helpful comments from Tim Browning
and the anonymous referees.
\end{ack}

\section{Background and notation}
\label{back}

A quadratic form in $r$ variables $Q(\vec{x})$ is integer-valued if
it can be written in the form $Q(\vec{x}) = \frac{1}{2} \vec{x}^{T} A \vec{x}$,
where $A$ is a symmetric $r \times r$ matrix with integer entries, and
even diagonal entries. The matrix $A$ is called the Gram matrix of
$Q$.  The quadratic form $Q$ is called positive-definite if $Q(\vec{x}) \geq 0$
for all $\vec{x} \in \R^{r}$ with equality if and only if $\vec{x} =
\vec{0}$. The discriminant of $Q$ is the determinant of $A$, and the level
of $Q$ is the smallest positive integer $N$ so that $N A^{-1}$ has
integer entries and even diagonal entries.

Let $\H = \{ x + iy : x, y \in \R, y > 0 \}$ denote the upper half plane.
If $k$ and $N$ are positive integers, and $\chi$ is a Dirichlet
character mod $N$, let $M_{k}(\Gamma_{0}(N), \chi)$ denote the vector
space of modular forms (holomorphic on $\H$ and at the
cusps) of weight $k$ that transform according to
\[
  f\left(\frac{az+b}{cz+d}\right) = \chi(d) (cz+d)^{k} f(z)
\]
for $\left[ \begin{matrix} a & b \\ c & d \end{matrix} \right] \in
\Gamma_{0}(N)$, the subgroup of $\SL_{2}(\Z)$ consisting of matrices
whose bottom left entry is a multiple of $N$. Let $S_{k}(\Gamma_{0}(N), \chi)$
denote the subspace of cusp forms. If $\lambda$ is an integer, let
$M_{\lambda + \frac{1}{2}}(\Gamma_{0}(4N), \chi)$ denote the vector space of
holomorphic half-integer weight modular forms that transform according to
\[
  g\left(\frac{az+b}{cz+d}\right) = \chi(d) \legen{c}{d}^{2 \lambda + 1}
  \epsilon_{d}^{-1 - 2 \lambda} (cz+d)^{\lambda + \frac{1}{2}} g(z)
\]
for all $\left[ \begin{matrix} a & b \\ c & d \end{matrix} \right] \in
\Gamma_{0}(4N)$. Here $\legen{c}{d}$ is the usual Jacobi symbol
if $d$ is odd and positive and $c \ne 0$. We define $\legen{0}{\pm 1} = 1$
and
\[
  \legen{c}{d} = \begin{cases}
    \legen{c}{|d|} & \text{ if } d < 0 \text{ and } c > 0,\\
    -\legen{c}{|d|} & \text{ if } d < 0 \text{ and } c < 0.
\end{cases}
\]
Finally $\epsilon_{d}$ is $1$ if $d \equiv 1 \pmod{4}$ and $i$ if $d \equiv 3
\pmod{4}$. Let $S_{\lambda + \frac{1}{2}}(\Gamma_{0}(4N), \chi)$ denote
the subspace of cusp forms.

For an integer-valued quadratic form $Q$, let $r_{Q}(n) = \#\{ \vec{x}
\in \Z^{r} : Q(\vec{x}) = n \}$. The theta series of $Q$ is the
generating function
\[
  \theta_{Q}(z) = \sum_{n=0}^{\infty} r_{Q}(n) q^{n}, \quad q = e^{2 \pi i z}.
\]
When $r$ is even, Theorem 10.9 of \cite{Iwa} shows that $\theta_{Q}(z)
\in M_{r/2}(\Gamma_{0}(N), \chi_{D})$, where $D = (-1)^{r/2} \det A$. If $r$ is
odd, Theorem 10.8 of \cite{Iwa} gives
that $\theta_{Q}(z) \in M_{r/2}(\Gamma_{0}(2N), \chi_{2 \det A})$. Here and
throughout, $\chi_{D}$ denotes the Kronecker character of the field
$\Q(\sqrt{D})$. We may
decompose $\theta_{Q}(z)$ as
\[
  \theta_{Q}(z) = E(z) + C(z)
\]
where $E(z) = \sum_{n=0}^{\infty} a_{E}(n) q^{n}$ is an Eisenstein series,
and $C(z) = \sum_{n=1}^{\infty} a_{C}(n) q^{n}$ is a cusp form.

If $Q$ is an integer-valued positive definite quadratic form $Q$, one can
associate to $Q$ a lattice $L$ (and vice versa) as follows. We let
$L = \Z^{r}$ and define an inner product on $L$ by
\[
  \langle \vec{x}, \vec{y} \rangle = \frac{1}{2} \left(Q(\vec{x} + \vec{y})
  - Q(\vec{x}) - Q(\vec{y})\right).
\]
If $\vec{x} \in L$, then $\langle \vec{x}, \vec{x} \rangle = Q(\vec{x})$ is an
integer, however arbitrary inner products $\langle x, y \rangle$ with
$\vec{x}, \vec{y} \in L$ need not be integral. Suppose that $R$ is an
integer-valued quadratic form in $m \leq r$ variables $y_{1}, y_{2},
\ldots, y_{m}$. Then $Q$ represents $R$ if there are linear forms
$L_{1}, L_{2}, \ldots, L_{r}$ in the $y_{i}$ with integer coefficients so that
\[
  Q(L_{1}, L_{2}, \ldots, L_{r}) = R.
\]
It is easy to see that this occurs if and only if there is a dimension $m$
sublattice $L' \subseteq L$ so that $L'$ is isometric to the lattice corresponding to $R$.

Let $\Z_{p}$ be the ring of $p$-adic integers. We say that $Q$ locally
represents the non-negative integer $m$ if for all primes $p$ there is
a vector $\vec{x}_{p} \in \Z_{p}^{r}$ so that $Q(\vec{x}_{p}) = m$. We
say that $m$ is represented by $Q$ if there is a vector $\vec{x} \in
\Z^{r}$ with $Q(\vec{x}) = m$.

For a quadratic form $Q$, we let $\Gen(Q)$ denote the finite collection of
quadratic forms $R$ so that $R$ is equivalent to $Q$ over $\Z_{p}$ for
all primes $p$. From the work of Siegel \cite{Siegel} it is known that
we can express the Eisenstein series $E(z)$ as a weighted sum over the
genus. In particular,
\begin{equation}
\label{genusaverage}
  E(z) = \frac{\sum_{R \in \Gen(Q)} \frac{\theta_{R}(z)}{\# \Aut(R)}}
{\sum_{R \in \Gen(Q)} \frac{1}{\# \Aut(R)}}.
\end{equation}
Moreover, the coefficients $a_{E}(m)$ of $E(z)$ can be expressed as a product
\[
  a_{E}(m) = \prod_{p \leq \infty} \beta_{p}(m)
\]
of local densities $\beta_{p}(m)$. We will make use of the algorithms of
Hanke \cite{Hanke} and the formulas of Yang \cite{Yang} for these local
densities.

If $Q$ is a quadratic form over $\Q_{p}$, $Q$ is equivalent to a diagonal form
\[
  a_{1} x_{1}^{2} + a_{2} x_{2}^{2} + \cdots + a_{r} x_{r}^{2}.
\]
The discriminant of $Q$ is defined to be $\prod_{i=1}^{r} a_{i}$, and
is well-defined up to a square in $\Q_{p}^{\times}$. We define the
$\epsilon$-invariant of $Q$ as in Serre \cite{Serre} by
\[
  \epsilon_{p}(Q) = \prod_{1 \leq i < j \leq r} (a_{i}, a_{j})_{p},
\]
where $(a,b)_{p}$ denotes the usual Hilbert symbol. Theorem 4.7
(pg. 39) of \cite{Serre} proves that two quadratic forms are
equivalent over $\Q_{p}$ if and only if they have the same rank $r$,
the same discriminant, and the same $\epsilon$-invariant.

If $Q$ is an integer-valued quadratic form and $p$ is a prime, we say
that $Q$ is anisotropic at $p$ if whenever $\vec{x} \in \Z_{p}^{r}$
and $Q(\vec{x}) = 0$, then $\vec{x} = 0$. If the rank
of $Q$ is $3$ or $4$, $Q$ has only finitely many anisotropic primes,
and if $Q$ is anisotropic at $p$, then $p | N$. When $r = 4$,
there is a unique $\Q_{p}$ equivalence class of forms that are anisotropic
at $p$. Such forms have a square discriminant in $\Q_{p}^{\times}$, and
$\epsilon$-invariant $\epsilon_{p}(Q) = -(-1,-1)_{p}$. If the rank of $Q$
is greater than or equal to $5$, $Q$ does not have any anisotropic primes.

We will briefly review the theory of integer weight newforms due to
Atkin, Lehner, and Li. If $d$ is a positive integer, the map
$f(z) | V(d) = f(dz)$ sends
$S_{k}(\Gamma_{0}(M), \chi)$ to $S_{k}(\Gamma_{0}(Md), \chi)$. For
forms $f, g \in S_{k}(\Gamma_{0}(N), \chi)$, define the
Petersson inner product
\[
  \langle f, g \rangle = \frac{3}{\pi [\SL_{2}(\Z) : \Gamma_{0}(N)]}
  \iint_{\H / \Gamma_{0}(N)} f(x+iy) \overline{g(x+iy)} y^{k} \, \frac{dx \, dy}{y^{2}}.
\]

For each prime $p$, there is a Hecke operator
$T(p) : S_{k}(\Gamma_{0}(N), \chi) \to S_{k}(\Gamma_{0}(N), \chi)$ given by
\[
  \left(\sum_{n=1}^{\infty} a(n) q^{n}\right) | T(p)
 = \sum_{n=1}^{\infty} \left(a(pn) + \chi(p) p^{k-1} a(n/p)\right) q^{n}.
\]
If $p$ is a prime with $\gcd(N,p) = 1$, then the adjoint of the Hecke
operator $T(p)$ under the Petersson inner product is $\overline{\chi}(p) T(p)$
(see Theorem 5.5.3 of \cite{DS}).

For $N$ fixed, let $S_{k}^{{\rm old}}(\Gamma_{0}(N), \chi)$ be the subspace
of $S_{k}(\Gamma_{0}(N), \chi)$ generated by
$S_{k}(\Gamma_{0}(M), \chi) | V(d)$ over all pairs
$(d,M)$ with $dM | N$, $\cond(\chi) | M$ and $M < N$.
Let $S_{k}^{{\rm new}}(\Gamma_{0}(N), \chi)$ be the orthogonal complement of
$S_{k}^{{\rm old}}(\Gamma_{0}(N), \chi)$ with respect to the Petersson inner
product.

A newform is a form $f \in S_{k}^{{\rm new}}(\Gamma_{0}(N), \chi)$ that
is a simultaneous eigenform of the operators $T(p)$ for all primes $p$,
and normalized so that if $f(z) = \sum_{n=1}^{\infty} a(n) q^{n}$,
then $a(1) = 1$. The space $S_{k}^{{\rm new}}(\Gamma_{0}(N), \chi)$ is spanned
by newforms. Deligne's theorem gives the bound
\[
  |a(n)| \leq d(n) n^{\frac{k-1}{2}}
\]
on the $n$th Fourier coefficient of any newform, where $d(n)$ is the
number of divisors of $n$. (In the case of $k = 2$, this result was
first established by Eichler, Shimura, and Igusa.) The adjoint formula for
the Hecke operators shows that if $f$ and $g$ are two distinct newforms,
then $\langle f, g \rangle = 0$. If $\cond(\chi)$
denotes the conductor of the Dirichlet character $\chi$ and $p$ is a
prime with $p | N$, then the $p$th coefficient of the newform $f$
satisfies
\begin{equation}
\label{sizeofap}
  |a(p)| = \begin{cases}
    p^{\frac{k-1}{2}} & \text{ if } \cond(\chi) \nmid N/p\\
    p^{\frac{k}{2} - 1} & \text{ if } p^{2} \nmid N \text{ and } \cond(\chi) | N/p\\
    0 & \text{ if } p^{2} | N \text{ and } \cond(\chi) | N/p.\\
\end{cases}
\end{equation}
(See Theorem 3 of \cite{Li}.)
Finally, define the operator $W_{N} : S_{k}^{{\rm new}}(\Gamma_{0}(N), \chi)
\to S_{k}^{{\rm new}}(\Gamma_{0}(N), \chi)$ by
\[
  f | W_{N} = N^{-k/2} z^{k/2} f\left(-\frac{1}{Nz}\right).
\]
We have $W_{N}^{2} = (-1)^{k}$.

If $\epsilon \in \{ \pm 1 \}$, define the subspace
$M_{k}^{\epsilon}(\Gamma_{0}(N), \chi)$ to be the set of forms
\[
  g(z) = \sum_{n=0}^{\infty} b(n) q^{n} \in M_{k}(\Gamma_{0}(N), \chi)
\]
with the property that $b(n) = 0$ if $\chi(n) = -\epsilon$, and let
$S_{k}^{\epsilon}(\Gamma_{0}(N), \chi) = M_{k}^{\epsilon}(\Gamma_{0}(N), \chi)
\cap S_{k}(\Gamma_{0}(N), \chi)$. Since the adjoint of $T(p)$ is
$\overline{\chi}(p) T(p)$, for a newform $f(z) = \sum_{n=1}^{\infty} a(n) q^{n}$ we have
$a(n) = \chi(n) \overline{a(n)}$ if $\gcd(n,N) = 1$. In the case
when $\chi$ is quadratic, and $\cond(\chi) = N$, the old subspace is
trivial, and
\[
  \dim S_{k}^{+}(\Gamma_{0}(N), \chi) = \dim S_{k}^{-}(\Gamma_{0}(N), \chi)
  = \frac{1}{2} \dim S_{k}(\Gamma_{0}(N), \chi)
\]
and the $\epsilon$-subspace is spanned by
$\{ f + \epsilon \overline{f} : f \text{ a newform } \}$,
where if $f(z) = \sum_{n=1}^{\infty} a(n) q^{n}$, then
$\overline{f}(z) = \sum_{n=1}^{\infty} \overline{a(n)} q^{n}$.

A newform $f$ of weight $k \geq 2$ is said to have complex multiplication
(or CM) if there is Hecke Gr\"ossencharacter $\xi$ that corresponds to it.
This means that there is an imaginary quadratic field $K = \Q(\sqrt{-D})$, a
nonzero ideal $\Lambda \subseteq O_{K}$, and a homomorphism $\xi$
from the group of all fractional ideals of $O_{K}$ relatively
prime to $\Lambda$ to $\C^{\times}$ so that
\[
   \xi(\alpha O_{K}) = \alpha^{k-1} \quad \text{ if }
  \alpha \equiv 1 \pmod{\Lambda},
\]
and so that
\[
  f(z) = \sum_{\mathfrak{a} \subseteq O_{K}} \xi(\mathfrak{a}) q^{N(\mathfrak{a})},
\]
where the sum if over all integral ideals $\mathfrak{a}$ of $O_{K}$
and $N(\mathfrak{a}) = \# (O_{K}/\mathfrak{a})$ denotes the norm of
$\mathfrak{a}$. For more details about Hecke Gr\"ossencharacters, see
Chapter 12 of \cite{Iwa}.

\section{Rankin-Selberg $L$-functions}
\label{RS}

If $Q$ is a positive-definite, quaternary, integer-valued quadratic form, then
$\theta_{Q}(z) = \sum_{n=0}^{\infty} r_{Q}(n) \in M_{2}(\Gamma_{0}(N), \chi)$
for some positive integer $N$, and Dirichlet character $\chi$. Let
$\theta_{Q}(z) = E(z) + C(z)$ be the decomposition as the sum of an Eisenstein
series and a cusp form, where
\[
  C(z) = \sum_{n=1}^{\infty} a_{C}(n) q^{n} \in S_{2}(\Gamma_{0}(N), \chi).
\]
Lower bounds on the coefficients $a_{E}(n)$ of $E(z)$ are given
by Hanke in \cite{Hanke} when $n$ is locally represented by $Q$ (provided $n$
has a priori bounded divisibility by any anisotropic primes) and are of the
form $a_{E}(n) \gg_{Q} n^{1 - \epsilon}$. We may decompose
\begin{equation}
\label{newformdecomp}
  C(z) = \sum_{M | N} \sum_{i=1}^{\dim S_{2}^{{\rm new}}(\Gamma_{0}(M), \chi)}
  \sum_{d} c_{d,i,M} g_{i,M} | V(d),
\end{equation}
where the $g_{i,M}$ are newforms of level $M$. Applying Deligne's bound,
we have that the $n$th Fourier coefficient of $g_{i,M} | V(d)$ is bounded by
\[
  d(n/d) \sqrt{n/d} \leq \frac{1}{\sqrt{d}} d(n) \sqrt{n}.
\]
Since we are interested in representations of odd integers, we define
$C_{Q}^{{\rm odd}}$ to be
\[
  C_{Q}^{{\rm odd}} := \sum_{M | N} \sum_{i} \sum_{d \text{ odd }} \frac{|c_{d,i,M}|}{\sqrt{d}},
\]
and we have that $|a_{C}(n)| \leq C_{Q}^{{\rm odd}} d(n) n^{1/2}$ for all odd $n$.

Combining the lower bound on $a_{E}(n)$ with the upper bound on $a_{C}(n)$
shows that $Q$ fails to represent only finitely many positive integers
that are locally represented by $Q$, and have bounded divisibility
by any anisotropic primes. We are interested in determining the dependence
on the form $Q$ of the constant $C_{Q}^{{\rm odd}}$, and the implied constant in the
estimate for $a_{E}(n) \gg_{Q} n^{1 - \epsilon}$. These bounds we obtain
will prove Theorem~\ref{bound} and will be the basis of one of the methods
we use in Section~\ref{pfof451} to prove the 451-Theorem.

For the remainder of this section, we assume that $Q$ is a
positive-definite, quaternary quadratic form whose discriminant $D$ is
a fundamental discriminant. This implies that $N = D$, and also that
$\chi$ is a primitive Dirichlet character modulo $N$. Then the old subspace
of $S_{2}(\Gamma_{0}(N), \chi)$ is trivial, and the decomposition above simply
becomes $C(z) = \sum_{i=1}^{u} c_{i} g_{i}(z)$,
where $u = \dim S_{2}(\Gamma_{0}(N), \chi)$, and the $g_{i}(z)$ are
newforms in $S_{2}(\Gamma_{0}(N), \chi)$. Taking the Petersson inner
product of $C$ with itself, and using that $\langle g_{i}, g_{j}
\rangle = 0$ if $i \ne j$ implies that
\[
  \langle C, C \rangle = \sum_{i=1}^{u}
  |c_{i}|^{2} \langle g_{i}, g_{i} \rangle.
\]
Suppose that we have bounds $A$ and $B$ so that
$\langle C, C \rangle \leq A$ and $\langle g_{i}, g_{i} \rangle \geq B$ for all
$i$. Then, we have
\[
  \sum_{i=1}^{u} B |c_{i}|^{2} \leq A
\]
and so
\begin{equation}
\label{cqdef}
  C_{Q}^{{\rm odd}} = \sum_{i=1}^{u} |c_{i}|
  \leq \sqrt{u} \sqrt{\sum_{i=1}^{u} |c_{i}|^{2}}\\
  \leq \sqrt{\frac{A u}{B}},
\end{equation}
which follows by the Cauchy-Schwarz inequality. Hence, a bound on
$C_{Q}^{{\rm odd}}$ follows from an upper bound on $\langle C, C \rangle$ and a
lower bound on $\langle g_{i}, g_{i} \rangle$. We will derive bounds
on both of these quantities using the theory of Rankin-Selberg
$L$-functions.

Suppose that $f(z) = \sum_{n=1}^{\infty} a(n) q^{n}$ and $g(z) =
\sum_{n=1}^{\infty} b(n) q^{n}$ are cusp forms of weight $k$. Rankin \cite{Rankin}
and Selberg \cite{Selberg} independently developed their convolution
$L$-function
\[
  \sum_{n=1}^{\infty} \frac{a(n) b(n)}{n^{s+k-1}}
\]
and studied its analytic properties. The most relevant property is
that the residue of this $L$-function at $s = 1$ is essentially the Petersson
inner product $\langle f, g \rangle$. Some of the specific results that we will require
about Rankin-Selberg $L$-functions
were worked out by Li in \cite{Li2}.
\begin{thm}
Suppose that $N$ is a fundamental discriminant, $\chi$ is a quadratic Dirichlet
character with conductor $N$, and
$f, g \in S_{2}(\Gamma_{0}(N), \chi)$ are newforms with $L$-functions
\begin{align*}
  L(f,s) &= \prod_{p} (1 - \alpha_{p} p^{-s})^{-1} (1 - \beta_{p} p^{-s})^{-1}\\
  L(g,s) &= \prod_{p} (1 - \gamma_{p} p^{-s})^{-1} (1 - \delta_{p} p^{-s})^{-1}.
\end{align*}
For $p | N$, exactly one of the Euler factors of $L(f,s)$ and $L(g,s)$ is zero,
and we make the convention that $\beta_{p} = \delta_{p} = 0$. Then
\begin{align*}
L(f \otimes g, s) &= \prod_{p | N} (1 - \alpha_{p} \gamma_{p} p^{-s})^{-1} (1 - \overline{\alpha}_{p} \overline{\gamma}_{p} p^{-s})^{-1}
\cdot\\
& \prod_{p \nmid N} (1 - \alpha_{p} \gamma_{p} p^{-s})^{-1}
(1 - \alpha_{p} \delta_{p} p^{-s})^{-1}
(1 - \beta_{p} \gamma_{p} p^{-s})^{-1}
(1 - \beta_{p} \delta_{p} p^{-s})^{-1},\\
  L(\Ad^{2} f, s) &= \prod_{p}
  (1 - \alpha_{p}^{2} \chi(p) p^{-s})^{-1} (1 - p^{-s})^{-1} (1 - \beta_{p}^{2}
  \chi(p) p^{-s})^{-1}.
\end{align*}
These two $L$-functions are entire (with the possible exception of
a pole at $s = 1$ for $L(f \otimes g, s)$) and satisfy the functional equations
\begin{align*}
\Lambda(f \otimes g, s) &= N^{s} \pi^{-2s} \Gamma\left(\frac{s}{2}\right)
\Gamma\left(\frac{s+1}{2}\right)^{2} \Gamma\left(\frac{s+2}{2}\right) L(f \otimes g,
 s),\\
\Lambda(f \otimes g, s) &= \Lambda(f \otimes g, 1-s),\\
\Lambda(\Ad^{2} f, s) &= N^{s} \pi^{-3s/2} \Gamma\left(\frac{s+1}{2}\right)^{2}
\Gamma\left(\frac{s+2}{2}\right) L(\Ad^{2} f, s),\\
\Lambda(\Ad^{2} f, 1-s) &= \Lambda(\Ad^{2} f, s).\\
\end{align*}
We also have
\[
  \Res_{s=1} L(f \otimes \overline{f}, s) = \frac{8 \pi^{4}}{3}
  \left(\prod_{p | N} 1 + \frac{1}{p}\right) \langle f, f \rangle.
\]
\end{thm}
\begin{proof}
The holomorphy and functional equations above follow from
Theorem 2.2 of \cite{Li2}, and the residue formula follows from
Theorem 3.2 of \cite{Li2}. In the notation of Li, $M = M' = 1$,
$M'' = N$, and the set $P$ is empty. The statements about $L(\Ad^{2} f, s)$ follow
from the observations that $L(\Ad^{2} f, s) = \frac{1}{\zeta(s)} L(f \otimes
\overline{f}, s)$, and that $L(\Ad^{2} f, s)$ is also entire (by work of
Gelbart and Jacquet \cite{GJ}).
\end{proof}

Goldfeld, Hoffstein and Lieman (see the appendix to \cite{GHL}) show
that if $f$ is not a CM form, then $L(\Ad^{2} f, s)$ cannot have any
real zeroes close to $s = 1$. This in turn implies a lower bound on
$L(\Ad^{2} f, 1)$. Their proof involves calculations with the symmetric
fourth power $L$-function. We make their bounds completely explicit
and we start by computing the local factors at primes dividing the level
using the local Langlands correspondence.

A newform $g$ corresponds to a cuspidal automorphic representation
$\pi$ of $\GL_{2}(\A_{\Q})$ (see \cite{Langlands}, Chapter 7 for details).
Such a representation can be factored as
\[
  \pi = \otimes_{p \leq \infty} \pi_{p}
\]
where each $\pi_{p}$ is a representation of $\GL_{2}(\Q_{p})$. The local
Langlands correspondence gives a bijection
between the set of smooth, irreducible representations of
$\GL_{n}(\Q_{p})$ and degree $n$ complex representations of the
Weil-Deligne group $W_{\Q_{p}}'$. It was conjectured by Langlands in 1967,
proven in odd residue characteristic for $\GL_{2}$ by Jacquet and Langlands in 1970,
and proven for $\GL_{n}$ by Harris and Taylor
\cite{HT}. For more details see Section 10.3 of \cite{Langlands},
\cite{Kudla}, and \cite{BushHenn} for a
thorough discussion of the $GL(2)$ case.

Known instances of automorphic
lifting maps (including the adjoint square map $r : \GL_{2} \to \GL_{3}$
due to Gelbart and Jacquet \cite{GJ},
the Rankin-Selberg convolution $r : \GL_{2} \times \GL_{2} \to \GL_{4}$ due to
Ramakrishnan \cite{Rama}, and the symmetric fourth power map
$r : \GL_{2} \to \GL_{5}$ due to Kim \cite{Kim}) are constructions of
automorphic representations
\[
  \Pi = r(\pi) = \otimes_{p \leq \infty} \Pi_{p}
\]
where $\Pi_{p}$ is computed by mapping $\pi_{p}$ to a degree $2$ complex
representation $\rho_{p}$ of $W_{\Q_{p}}'$ via the local Langlands
correspondence, computing $r(\rho_{p})$ and mapping back to the automorphic
side (again by the local Langlands correspondence). Since
the local Langlands correspondence preserves local $L$-functions and conductors, to compute
these, it suffices to know the representations $r(\rho_{p})$.

\begin{prop}
Suppose that $N$ is a fundamental discriminant, $\chi$ is
a Dirichlet character with conductor $N$, and
$f$ is a newform without CM in $S_{2}(\Gamma_{0}(N), \chi)$ with $L$-function
$L(f,s) = \prod_{p} (1 - \alpha_{p} p^{-s})^{-1} (1 - \beta_{p} p^{-s})^{-1}$.
Define
\[
  L(\Sym^{4} f, s) = \prod_{p} (1 - \alpha_{p}^{4} p^{-s})^{-1}
  (1 - \alpha_{p}^{2} \chi(p) p^{-s})^{-1} (1 - p^{-s})^{-1}
  (1 - \alpha_{p}^{-2} \chi(p) p^{-s})^{-1} (1 - \alpha_{p}^{-4} p^{-s})^{-1}.
\]
This $L$-function is entire and satisfies the function equation
\begin{align*}
\Lambda(\Sym^{4} f, s) &= N^{s} \pi^{-5s/2} \Gamma\left(\frac{s}{2}\right)
  \Gamma\left(\frac{s+1}{2}\right) \Gamma\left(\frac{s+2}{2}\right)^{2}
  \Gamma\left(\frac{s+3}{2}\right) L(\Sym^{4} f, s)\\
\Lambda(\Sym^{4} f, 1-s) &= \Lambda(\Sym^{4} f, s).
\end{align*}
\end{prop}
\begin{rem}
In the case that $f$ does have CM, $L(\Sym^{4} f, s)$ has a pole at
$s = 1$ and the proof of Proposition~\ref{llow} below breaks down. This
is the source of ineffectivity in Theorem~\ref{bound}.
\end{rem}
\begin{rem}
One can obtain numerical confirmation of the result above by
checking the stated functional equations using the $L$-functions
package (available in PARI/GP, Magma and Sage) due to Tim Dokchitser
(see \cite{Dok}).
\end{rem}
\begin{proof}
Let $p$ be a prime dividing $N$ and let $\pi$ be the local representation of $\GL_{2}(\Q_{p})$ that occurs
as a constituent of $f$. Since the
$p$th Fourier coefficient of $f$ has absolute value
$p^{1/2}$ (by \eqref{sizeofap}), $\pi$ must be a principal series representation
$\pi(\epsilon, \chi_{p} \epsilon^{-1})$, where $\epsilon$ is an unramified character
of $\Q_{p}^{\times}$ and $\chi_{p}$ is the local component of the Dirichlet character
$\chi$ at $p$. (This follows from a comparison of the different options for the local
$L$-functions described in Chapter 6, Sections 25 and 26 of \cite{BushHenn}.)

Applying the local Langlands correspondence, it follows that $\pi$
corresponds to a representation $\rho$ of the Weil group that is a sum of two characters $\sigma \mapsto \delta_{1} \oplus \delta_{2}$.
The Weil group $W_{\Q_{p}}$ is the subgroup of $\Gal(\overline{\Q_{p}}/\Q_{p})$ consisting
of all elements restricting to some power of the Frobenius on $\overline{\F_{p}}$. It is a quotient of the Weil-Deligne group.
The local Langlands correspondence maps a character of $\Q_{p}^{\times}$ to a charater of $\Gal(\overline{\Q_{p}}/\Q_{p})$
using the reciprocity law homomorphism $c : \Gal(\overline{\Q_{p}}/\Q_{p}) \to \Q_{p}^{\times}$ of class field theory
$\chi \mapsto \chi \circ c$.
Hence, the representation of $W_{\Q_{p}}$ corresponding to $\pi(\epsilon, \psi \epsilon^{-1})$ is
$\rho_{1} \oplus \rho_{2}$, where $\rho_{1} = \epsilon \circ c$ and $\rho_{2} = \chi_{p} \epsilon^{-1} \circ c$.
Therefore, if $r$ is the symmetric fourth power map $r : \GL_{2} \to \GL_{5}$, we have that
\[
  r(\rho_{1} \oplus \rho_{2}) = \rho_{1}^{4} \oplus \rho_{1}^{3} \rho_{2} \oplus \rho_{1}^{2} \rho_{2}^{2}
  \oplus \rho_{1} \rho_{2}^{3} \oplus \rho_{2}^{4}.
\]
Since the $L$-function of a semisimple Weil-Deligne representation
$\rho : \Gal(\overline{\Q}_{p}/\Q_{p}) \to GL(V)$ is given by
\[
  \det\left(1 - p^{-s} \rho({\rm Frob}_{p}) | V^{I_{p}}\right),
\]
we have that for a character $\chi$, $L(\chi,s) = (1 - \chi({\rm Frob}_{p}) p^{-s})^{-1}$ if $\chi$ is unramified
and $L(\chi,s) = 1$ if $\chi$ is ramified. The stated formula for the local factors follows from the observation
that $\rho_{1}^{3} \rho_{2}$ and $\rho_{1} \rho_{2}^{3}$ are ramified, while the other three characters are unramified.
The characters $\rho_{1}^{3} \rho_{2}$ and $\rho_{1} \rho_{2}^{3}$ have the same conductor as that of $\rho_{2}$
(which is $p$ if $p > 2$, and is either $p^{2}$ or $p^{3}$ if $p = 2$). A simple
calculation shows that the product of the local signs over all primes
$p$ is equal to $1$. The global conductor is the product of the local conductors and is hence $N^{2}$. The gamma factors are known (see \cite{CM}).
\end{proof}

Now, we make effective the zero-free region due to Goldfeld, Hoffstein,
and Lieman from the appendix to \cite{GHL}. (See
Lemmas 2 and 3 of \cite{Rouse} for a version in the case that $f$ has level
one.)

\begin{prop}
Suppose $N$ is a fundamental discriminant, $\chi$ is a quadratic Dirichlet character
with conductor $N$, and $f \in S_{2}(\Gamma_{0}(N), \chi)$ is a newform without complex
multiplication. If $N \geq 44$, then $L(\Ad^{2} f, s)$ has no
real zeroes $s$ with
\[
  s > 1 - \frac{5 - 2 \sqrt{6}}{4 \log(N) - 11}.
\]
\end{prop}
\begin{proof}
Goldfeld, Hoffstein, and Lieman use the auxiliary degree 16 $L$-function
\[
  L(s) = \zeta(s)^{2} L(\Ad^{2} f, s)^{3} L(\Sym^{4} f, s).
\]
The gamma factor is
\[
  G(s) = N^{4s} \pi^{-16s/2} \Gamma\left(\frac{s}{2}\right)^{3}
\Gamma\left(\frac{s+1}{2}\right)^{7} \Gamma\left(\frac{s+2}{2}\right)^{5}
\Gamma\left(\frac{s+3}{2}\right),
\]
and the completed $L$-function $\Lambda(s) = G(s) L(s)$ has a meromorphic
continuation to $\C$ and satisfies the functional equation
$\Lambda(s) = \Lambda(1-s)$. The function $s^{2} (1-s)^{2} \Lambda(s)$ is an entire function of order $1$, and so we let
\[
  s^{2} (1-s)^{2} \Lambda(s) = e^{A + Bs} \prod_{\rho} \left(1 - \frac{s}{\rho}\right) e^{s/\rho}
\]
be its Hadamard product expansion. Taking the logarithmic derivative
gives
\begin{equation}
\label{logderiv}
  \sum_{\rho} \frac{1}{s - \rho} + \frac{1}{\rho}
  = \frac{2}{s} + \frac{2}{s-1} + \frac{G'(s)}{G(s)} + \frac{L'(s)}{L(s)} - B.
\end{equation}
We take the real part of both sides. Part 3 of Proposition 5.7 of
\cite{IK} gives that $\RE{B} = -\sum_{\rho} \RE{\frac{1}{\rho}}$. The
Dirichlet coefficients of $-L'(s)/L(s)$ are non-negative, and this
implies that $L'(s)/L(s) < 0$ if $s > 1$ is real. Taking the real part
of \eqref{logderiv} gives that
\[
  \sum_{\rho} \RE{\frac{1}{s - \rho}}
  \leq \frac{2}{s} + \frac{2}{s-1} + \frac{G'(s)}{G(s)}.
\]
We have
\begin{align*}
  \frac{G'(s)}{G(s)} &= 4 \log(N) - 8 \log(\pi)
  + \frac{1}{2} \left[3 \psi(s/2) + 7 \psi((s+1)/2) + 5 \psi((s+2)/2)
  + \psi((s+3)/2) \right],
\end{align*}
where $\psi(s) = \frac{\Gamma'(s)}{\Gamma(s)}$. Since $\psi(s)$ is
an increasing function of $s$, we have that
$\frac{G'(s)}{G(s)} \leq 4 \log(N) - 13$ if $s \leq 1.11$.

We set $s = 1 + \alpha$ where $0 \leq \alpha \leq 0.05$ will be chosen later.
If $\beta$ is a real zero of $L(\Ad^{2} f, s)$, then it is a triple
zero of $L(s)$, and this means that
\[
  \frac{3}{\alpha + 1 - \beta} \leq \frac{2}{\alpha + 1}
  + \frac{2}{\alpha} + \frac{G'(1+\alpha)}{G(1+\alpha)}
  \leq \frac{2}{\alpha} + (4 \log(N) - 11).
\]
Choosing $\alpha$ optimally gives that $1 - \beta \geq
\frac{5 - 2 \sqrt{6}}{4 \log(N) - 11}$, provided the corresponding
value of $s$ is less than $1.11$. This occurs for $N \geq 44$, and
shows that
\[
  \beta \leq 1 - \frac{5 - 2 \sqrt{6}}{4 \log(N) - 11}.
\]
\end{proof}
We now translate the above result into a lower bound on
$L(\Ad^{2} f, 1)$ by a similar argument to that in Lemma 3 of \cite{Rouse}.
\begin{prop}
\label{llow}
Suppose that $N$ is a fundamental disriminant, $\chi$ is a quadratic Dirichlet character
with conductor $N$, and $f$ is a newform in $S_{2}(\Gamma_{0}(N), \chi)$ that
does not have complex multiplication. Then
\[
  L(\Ad^{2} f, 1) > \frac{1}{26 \log(N)}.
\]
\end{prop}
\begin{proof}
We consider
\[
  L(f \otimes \overline{f}, s) = \sum_{n=1}^{\infty} \frac{a(n)}{n^{s}}.
\]
The $p$th Euler factor of $L(f \otimes \overline{f}, s)$ for $p \nmid N$ is
\[
  (1 - \alpha_{p}^{2} \chi(p) p^{-s})^{-1} (1 - p^{-s})^{-2} (1 - \alpha_{p}^{-2} \chi(p) p^{-s})^{-1}
  = \sum_{r=0}^{\infty} \frac{1}{p^{rs}} \sum_{k=0}^{\lfloor r/2 \rfloor} \left(\sum_{l = k}^{r-k}
    (\sqrt{\chi(p)} \alpha_{p})^{r-2l}\right)^{2}.
\]
Each of the inner sums over $l$ are real and so the coefficient of $p^{-rs}$ is non-negative for all $r$.
Also, when $r$ is even, the term with $l = r/2$ contributes $1$ and so the coefficient of $p^{-rs}$
is $\geq 1$ when $r$ is even. Similar conclusions hold for $p | N$ where the local factor is
$(1 - p^{-s})^{-2} = \sum_{n=0}^{\infty} \frac{n+1}{p^{ns}}$. It follows that $a(n) \geq 0$ and $a(n^{2}) \geq 1$ hold for all positive integers $n$.

Let $\beta = 1 - \frac{5 - 2 \sqrt{6}}{4 \log(N) - 11}$ and assume
that $N$ is large enough that $\beta \geq 3/4$. Set
$x = N^{A}$, where we let $A$ be a parameter that we will choose optimally
at the end of the argument. We consider
\[
  I = \frac{1}{2 \pi i} \int_{2 - i \infty}^{2 + i \infty}
  \frac{L(f \otimes \overline{f}, s + \beta) x^{s} \, ds}{s \prod_{k=2}^{10}
(s+k)}.
\]
We have that
\[
  \frac{1}{2 \pi i} \int_{2-i\infty}^{2 + i \infty}
  \frac{x^{s} \, ds}{s \prod_{k=2}^{10} (s+k)}
  = \begin{cases}
    \frac{(x+9) (x-1)^{9}}{10! x^{10}}, & \text{ if } x > 1\\
    0, & \text{ if } x < 1.
\end{cases}
\]
Therefore,
\begin{align*}
  I &= \frac{1}{2 \pi i}
  \int_{2 - i \infty}^{2 + i \infty} \frac{L(f \otimes \overline{f}, s + \beta)}{s \prod_{k=2}^{10} (s+k)} = \sum_{n \leq x}
  \frac{a(n) (x/n + 9) (x/n - 1)^{9}}{10! n^{\beta} (x/n)^{10}}\\
  &\geq \frac{1}{10!} \sum_{n^{2} \leq x}
  \frac{(x/n^{2} + 9) (x/n^{2} - 1)^{9}}{n^{2} (x/n^{2})^{10}}.
\end{align*}
Since the function $g(z) = \frac{(z+9) (z-1)^{9}}{z^{10}}$ is increasing
for $z > 1$, the above expression is increasing as a function of $x$.
If $x \geq 3989$, then $I \geq \frac{1.6}{10!}$, and
if $x \geq 330775$, then $I \geq \frac{1.64}{10!}$.

Now, we move the contour to $\RE{s} = \alpha$, where $\alpha = -3/2 - \beta$.
There are poles at $s = 1 - \beta$, $s = 0$, and $s = -2$. We get
\begin{align*}
  I =& \frac{1}{2 \pi i} \int_{\alpha - i \infty}^{\alpha + i \infty}
  \frac{L(f \otimes \overline{f}, s + \beta) x^{s} \, ds}{s \prod_{k=2}^{10}
  (s + k)} +
  \frac{L(\Ad^{2} f, 1) x^{1 - \beta}}{(1 - \beta) \prod_{k=2}^{10} (1 - \beta + k)}\\
  &+ \frac{L(f \otimes \overline{f}, \beta)}{10!} - \frac{L(f \otimes \overline{f}, -2 + \beta) x^{-2}}{2 \cdot 8!}.
\end{align*}
There are no zeroes of $L(\Ad^{2} f, s)$ to the right of $\beta$ and so
$L(\Ad^{2} f, \beta) \geq 0$. Since $\zeta(\beta) < 0$, it follows
that $L(f \otimes \overline{f}, \beta) \leq 0$. Since the sign of the functional
equation of $L(f \otimes \overline{f}, s)$ is 1, it follows that
there are an even number of real zeroes in the interval $(0,1)$ and hence
$L(f \otimes \overline{f}, 0) < 0$. The only zeroes with $s < 0$ are trivial
zeroes, and a simple zero occurs at $s = -1$. Thus,
$L(f \otimes \overline{f}, -2 + \beta) > 0$ and so
\[
  I - \frac{1}{2 \pi i} \int_{\alpha - i \infty}^{\alpha + i \infty}
  \frac{L(f \otimes \overline{f}, s + \beta) x^{s} \, ds}{s \prod_{k=2}^{10}
  (s + k)} \leq
  \frac{L(\Ad^{2} f, 1) x^{1 - \beta}}{(1 - \beta) \prod_{k=2}^{10} (1 - \beta + k)}
\]

Now, we apply the functional equation for $L(f \otimes \overline{f}, s)$. It
gives that
\[
  \left|L\left(f \otimes \overline{f}, -\frac{3}{2} + it\right)\right|
  = \frac{N^{4}}{(4 \pi)^{8}} |1 + 2it|^{4} |3 + 2it|^{3} |5 + 2it|
  \left|L\left(f \otimes \overline{f}, \frac{5}{2} - it\right)\right|.
\]
We have that $|L(f \otimes \overline{f}, \frac{5}{2} - it)| \leq
\zeta(5/2)^{4}$. We use this to derive the bound
\begin{align*}
  & \frac{1}{2 \pi} \int_{\alpha - i \infty}^{\alpha + i \infty}
  \left|\frac{L(f \otimes \overline{f}, s + \beta) x^{s}}{s \prod_{k=2}^{10} (s+k)}\right| \, ds\\
  &\leq \frac{N^{4 + A(-3/2 - \beta)} \zeta(5/2)^{4}}{2^{17} \pi^{9}}
  \int_{-\infty}^{\infty} \frac{|1 + 2it|^{4} |3 + 2it|^{3} |5 + 2it|}{|-3/2-\beta - it| \prod_{k=2}^{10} |k-3/2-\beta + it|} \, dt\\
  &\leq \frac{N^{4 + A(-3/2-\beta)} \zeta(5/2)^{4}}{2^{17} \pi^{9}}
  \int_{-\infty}^{\infty} \frac{|1 + 2it|^{4} |3 + 2it|^{3} |5+2it|}{|1/4 +it| |9/4 + it| \prod_{k=3}^{10} |k - 5/2 + it|} \, dt.
\end{align*}
Numerical computation gives the bound
\[
\int_{-\infty}^{\infty} \frac{|1 + 2it|^{4} |3 + 2it|^{3} |5+2it|}{|1/4 +it| |9/4 + it| \prod_{k=3}^{10} |k - 5/2 + it|} \, dt
\leq 2.776686,
\]
and this gives
\[
  \frac{1}{2 \pi} \int_{\alpha - i \infty}^{\alpha + i \infty}
  \left|\frac{L(f \otimes \overline{f}, s + \beta) x^{s}}{s \prod_{k=2}^{10} (s+k)}\right| \, ds \leq N^{4 + A(-3/2 - \beta)} \cdot \frac{8.35176 \cdot 10^{-3}}{10!}
\]
Out of this, we get the lower bound
\[
  L(\Ad^{2} f, 1) \geq (1 - \beta) \left( \frac{c}{N^{A (1 - \beta)}}
 - \frac{d}{N^{(5/2)A - 4}} \right),
\]
where $c = 1.6$ or $1.64$ depending on whether $3989 \leq x < 330775$
or $x \geq 330775$. If we choose $A = 8/5$ we get
\[
  L(\Ad^{2} f, 1) \geq \frac{1}{26 \log(N)}.
\]
For computational purposes, we use the optimal choice of $A$, namely
\begin{equation}
\label{optimal}
  A = \frac{1}{\beta + 3/2} \left[ 4 -
  \frac{\log(1 - \beta) + \log(c) - \log(d) - \log(5/2)}{\log(N)} \right].
\end{equation}
These bounds suffice when $N \geq 167$. For each of the newforms
of level $\leq 166$ satisfying the hypotheses, we compute their Fourier
coefficients using Magma and verify the claimed bound using
Proposition~\ref{petform} (whose proof does not depend on the present result).
\end{proof}

The above proposition implies a lower bound on the Petersson norm of
a newform $f$. We now turn to the problem of bounding from above
the Petersson norm $\langle C, C \rangle$. We will give a formula for
$\langle C, C \rangle$ using the functional equation for Rankin-Selberg
$L$-functions, and this formula will be used in subsequent sections
to prove the 451-Theorem and Theorem~\ref{bound}. First, we give a Dirichlet
series representation for the Rankin-Selberg $L$-function $L(f \otimes g, s)$.

\begin{lem}
\label{rankinformula}
Let $N$ be a fundamental discriminant and $\chi$ be a quadratic
Dirichlet character with conductor $N$. If
$f, g \in S_{2}(\Gamma_{0}(N), \chi)$ are newforms with
\[
  f(z) = \sum_{n=1}^{\infty} a(n) q^{n}, \text{ and }
  g(z) = \sum_{n=1}^{\infty} b(n) q^{n},
\]
then
\[
  L(f \otimes g, s) = \sum_{n=1}^{\infty} \left(\sum_{\substack{m | n \\ n/m \text{ is a square }}} \frac{2^{\omega(\gcd(m,N))} \RE{a(m) b(m)}}{m} \right)\frac{1}{n^{s}}.
\]
Here for a positive integer $m$, $\omega(m)$ denotes the number of distinct prime factors of $m$.
\end{lem}
\begin{proof}
Equation (13.1) of \cite{Iwa} states that if
\[
  \sum_{n=1}^{\infty} \frac{c(n)}{n^{s}} = \prod_{p} (1 - \alpha_{p} p^{-s})^{-1} (1 - \beta_{p} p^{-s})^{-1}, \text{ and } \quad
  \sum_{n=1}^{\infty} \frac{e(n)}{n^{s}} = \prod_{p} (1 - \gamma_{p} p^{-s})^{-1} (1 - \delta_{p} p^{-s})^{-1},
\]
then
\begin{align*}
& \sum_{n=1}^{\infty} \frac{c(n) e(n)}{n^{s}}
= \prod_{p} (1 - \alpha_{p} \beta_{p} \gamma_{p} \delta_{p} p^{-2s})\\
& \prod_{p} (1 - \alpha_{p} \gamma_{p} p^{-s})^{-1}
(1 - \alpha_{p} \delta_{p} p^{-s})^{-1} (1 - \beta_{p} \gamma_{p} p^{-s})^{-1}
(1 - \beta_{p} \delta_{p} p^{-s})^{-1}.
\end{align*}
If we take $c(n) = a(n)/\sqrt{n}$ and $e(n) = b(n)/\sqrt{n}$, and
$L(f,s) = \prod_{p} (1 - \alpha_{p} p^{-s})^{-1} (1 - \beta_{p} p^{-s})^{-1}$
and $L(g,s) = \prod_{p} (1 - \gamma_{p} p^{-s})^{-1} (1 - \delta_{p} p^{-s})^{-1}$,
it follows that
\begin{align*}
  & \prod_{p \nmid N} (1 - \alpha_{p} \gamma_{p} p^{-s})^{-1}
  (1 - \alpha_{p} \delta_{p} p^{-s})^{-1} (1 - \beta_{p} \gamma_{p} p^{-s})^{-1}
  (1 - \beta_{p} \delta_{p} p^{-s})^{-1}\\
  &= \prod_{p \nmid N} (1 - p^{-2s})^{-1} \sum_{n \text{ coprime to } N}
  \frac{a(n) b(n)}{n^{s+1}}.
\end{align*}
For $p | N$, we again make the convention that $\beta_{p} = \delta_{p} = 0$. Thus
\begin{equation}
\label{locatp}
  (1 - \alpha_{p} \gamma_{p} p^{-s})^{-1}
  = \sum_{k=0}^{\infty} \frac{a(p^{k}) b(p^{k})}{p^{k(s+1)}}.
\end{equation}
The local factor of $L(f \otimes g, s)$ at $p$ is
$(1 - \alpha_{p} \gamma_{p} p^{-s})^{-1} (1 - \overline{\alpha_{p}} \overline{\gamma_{p}} p^{-s})^{-1}$. Multiplying \eqref{locatp} by its conjugate, we get
\begin{align*}
  (1 - \alpha_{p} \gamma_{p} p^{-s})^{-1} (1 - \overline{\alpha_{p}}
  \overline{\gamma_{p}} p^{-s})^{-1} &= \sum_{k=0}^{\infty}
  \frac{1}{p^{k(s+1)}}
\sum_{i=0}^{k} a(p^{i}) b(p^{i})
\overline{a(p)^{k-i} b(p)^{k-i}}\\
  &= (1-p^{-2s})^{-1} \left(1 +
  2 \sum_{k=1}^{\infty} \frac{\RE{a(p^{k}) b(p^{k})}}{p^{k(s+1)}} \right).
\end{align*}
Taking the product of the local factors over all primes $p$ gives us the
desired formula.
\end{proof}

If $C_{1}$ and $C_{2}$ are arbitrary cusp forms in $S_{2}(\Gamma_{0}(N), \chi)$,
we define $L(C_{1} \otimes C_{2}, s)$ as follows. Write
\[
  C_{1}(z) = \sum_{i=1}^{u} c_{i} g_{i}(z) \text{ and }
  C_{2}(z) = \sum_{j=1}^{u} d_{j} g_{j}(z),
\]
where the $g_{i}(z)$, $1 \leq i \leq u$ are the newforms. Then, let
\[
  L(C_{1} \otimes C_{2}, s) = \sum_{i=1}^{u} \sum_{j=1}^{u}
  c_{i} d_{j} L(g_{i} \otimes g_{j}, s).
\]
The formula from Lemma~\ref{rankinformula} is not, in general, bilinear,
and so it cannot equal $L(C_{1} \otimes C_{2}, s)$ for all pairs
$C_{1}, C_{2} \in S_{2}(\Gamma_{0}(N), \chi)$. The next result is that
the formula is valid, provided both $C_{1}$ and $C_{2}$ are in
$S_{2}^{+}(\Gamma_{0}(N), \chi)$ or $S_{2}^{-}(\Gamma_{0}(N), \chi)$.

\begin{lem}
Let $N$ be a fundamental discriminant and $\chi$ be a quadratic
Dirichlet character with conductor $N$. Suppose that $f, g \in S_{2}^{\epsilon}(\Gamma_{0}(N), \chi)$ where
$\epsilon \in \{ \pm 1 \}$ and
\[
  f(z) = \sum_{n=1}^{\infty} a(n) q^{n}, \quad
  g(z) = \sum_{n=1}^{\infty} b(n) q^{n},
\]
with $a(n), b(n) \in \R$ for all $n$.
Then
\[
  L(f \otimes g, s) = \sum_{n=1}^{\infty}
  \left( \sum_{\substack{m | n \\ n/m \text{ is a square }}}
  \frac{2^{\omega(\gcd(m,N))} a(m) b(m)}{m} \right) \frac{1}{n^{s}}.
\]
Moreover, if
\[
  \Lambda(f \otimes g, s) = N^{s} \pi^{-2s} \Gamma\left(\frac{s}{2}\right)
\Gamma\left(\frac{s+1}{2}\right)^{2} \Gamma\left(\frac{s+2}{2}\right) L(f \otimes g, s),
\]
then $\Lambda(f \otimes g, s) = \Lambda(f \otimes g, 1-s)$, and we have
\[
\Res_{s=1} L(f \otimes g, s) =
\frac{8 \pi^{4}}{3} \left(\prod_{p | N} 1 + \frac{1}{p}\right) \langle f,
g \rangle.
\]
\end{lem}
\begin{proof}
All of the statements in the theorem are $\R$-bilinear. For this
reason, it suffices to prove them on a collection of basis elements
for $S_{2}^{\epsilon}(\Gamma_{0}(N), \chi) \cap \R[[q]]$: those of the form
$h + \overline{h}$ if $\epsilon = 1$ and $i (h - \overline{h})$
if $\epsilon = -1$. Suppose that $h_{1}$ and $h_{2}$
are newforms with
\[
  h_{1}(z) = \sum_{n=1}^{\infty} a(n) q^{n}, \quad
  h_{2}(z) = \sum_{n=1}^{\infty} b(n) q^{n},
\]
and set $i_{1}(z) = h_{1}(z) + \overline{h_{1}}(z)$
and $i_{2}(z) = h_{2}(z) + \overline{h_{2}}(z)$ in the case that
$\epsilon = 1$ and $i_{1}(z) = i (h_{1}(z) - \overline{h_{1}}(z))$
and $i_{2}(z) = i (h_{2}(z) - \overline{h_{2}}(z))$ in the
case that $\epsilon = -1$. A straightforward calculation shows that
in both cases,
\[
  L(i_{1} \otimes i_{2}, s) = \epsilon L(h_{1} \otimes h_{2}, s) +
  L(\overline{h_{1}} \otimes h_{2}, s) + L(h_{1} \otimes
  \overline{h_{2}}, s) + \epsilon L(\overline{h_{1}} \otimes \overline{h_{2}}, s).
\]
The formula in Lemma~\ref{rankinformula} shows that for newforms
$f$ and $g$,
$L(\overline{f} \otimes \overline{g}, s) = L(f \otimes g, s)$ and so we have
\[
  L(i_{1} \otimes i_{2}, s) =
  2\epsilon L(h_{1} \otimes h_{2}, s) + 2 L(\overline{h_{1}} \otimes h_{2}, s).
\]
This equality proves all of the claimed results, with the exception
of the Dirichlet series representation for $L(i_{1} \otimes i_{2}, s)$.

If $\epsilon = 1$, we have that the numerator of a term in the inner sum of
$L(i_{1} \otimes i_{2}, s)$ is
\begin{align*}
  & 2^{\omega(\gcd(n,N))}
  \left(2 \RE{a(n) b(n)} + 2 \RE{\overline{a(n)} b(n)}\right)\\
  &= 2^{\omega(\gcd(n,N))} (a(n) b(n) + \overline{a(n) b(n)})
  + 2^{\omega(\gcd(n,N))} (\overline{a(n)} b(n) + a(n) \overline{b(n)})\\
  &= 2^{\omega(\gcd(n,N))} (a(n) + \overline{a(n)}) (b(n) +
\overline{b(n)}).
\end{align*}
If $\epsilon = -1$, we have
\begin{align*}
  & 2^{\omega(\gcd(n,N))} \left(-2 \RE{a(n) b(n)} + 2 \RE{\overline{a(n)} b(n)} \right)\\
  &= 2^{\omega(\gcd(n,N))} (-a(n) b(n) - \overline{a(n) b(n)})
  + 2^{\omega(\gcd(n,N))} (\overline{a(n)} b(n) + a(n) \overline{b(n)})\\
  &= 2^{\omega(\gcd(n,N))} (i a(n) - i \overline{a(n)})
  (i b(n) - i \overline{b(n)}).
\end{align*}
It follows that if $i_{1}(z) = \sum_{n=1}^{\infty} c(n) q^{n}$
and $i_{2}(z) = \sum_{n=1}^{\infty} e(n) q^{n}$, then
\[
  L(i_{1} \otimes i_{2}, s)
  = \sum_{n=1}^{\infty}
  \left( \sum_{\substack{m | n \\ n/m \text{ is a square}}} \frac{2^{\omega(\gcd(m,N))} c(m) e(m)}{m} \right) \frac{1}{n^{s}},
\]
which completes the proof.
\end{proof}

\begin{rem}
If $f \in S_{2}^{+}(\Gamma_{0}(N), \chi)$ and
$g \in S_{2}^{-}(\Gamma_{0}(N), \chi)$ have real Fourier coefficients, one can
see from the definition that $L(f \otimes g, s) = 0$, while the formula
from Lemma~\ref{rankinformula} is typically nonzero. This shows that
one cannot use the formula in Lemma~\ref{rankinformula} in all cases.
\end{rem}

Finally, we give a formula for $\langle C, C \rangle$ under the assumption
that $C \in S_{2}^{\epsilon}(\Gamma_{0}(N), \chi)$. We follow the approach
in \cite{Dok}. To state our result, let $K_{\nu}(z)$ denote the usual
$K$-Bessel function of order $\nu$.

\begin{prop}
\label{petform}
Let $N$ be a fundamental discriminant and $\chi$ be a quadratic Dirichlet character
with conductor $N$.
Suppose that $C(z) = \sum_{n=1}^{\infty} a(n) q^{n} \in S_{2}^{\epsilon}(\Gamma_{0}(N), \chi)$ for $\epsilon \in \{ \pm 1 \}$. Let
\[
  \psi(x) = -\frac{6}{\pi} x K_{1}(4 \pi x) + 24x^{2} K_{0}(4 \pi x).
\]
Then,
\[
  \langle C, C \rangle
  = \frac{1}{[\SL_{2}(\Z) : \Gamma_{0}(N)]} \sum_{n=1}^{\infty}
  \frac{2^{\omega(\gcd(n,N))} a(n)^{2}}{n} \sum_{d=1}^{\infty} \psi\left(d \sqrt{\frac{n}{N}}\right).
\]
\end{prop}
\begin{proof}
Define (as in \cite{Dok}, pg. 139) the function
\[
  \Theta(t) = \sum_{n=1}^{\infty} b(n) \phi\left(\frac{nt}{N}\right),
\]
where $b(n)$ is the $n$th Dirichlet coefficient of $L(C \otimes C, s)$, namely
\[
  b(n) = \sum_{\substack{m | n \\ \frac{n}{m} \text{ is a square }}}
  \frac{2^{\omega(\gcd(m,N))} a(m)^{2}}{m},
\]
and $\phi$ is the inverse Mellin transform of the gamma factor
$\pi^{-2s} \Gamma\left(\frac{s}{2}\right) \Gamma\left(\frac{s+1}{2}\right)^{2}
\Gamma\left(\frac{s+2}{2}\right)$. Then, $\Theta(t)$ is the inverse Mellin
transform of $\Lambda(C \otimes C, s)$. Using the functional equation
and shifting the contour to the left gives the formula
\begin{equation}
\label{funceq}
  \Theta\left(\frac{1}{t}\right) = t\Theta(t) + r(t-1)
\end{equation}
where $r = \Res_{s=1} \Lambda(C \otimes C, s) =
-\Res_{s=0} \Lambda(C \otimes C, s)$. Differentiating \eqref{funceq}
and setting $t = 1$ gives
\[
  -\Theta(1) - 2\Theta'(1) = r.
\]
Equation (10.43.19) of \cite{DLMF} gives the Mellin transform
\[
  \int_{0}^{\infty} t^{\mu - 1} K_{\nu}(t)
  \, dt = 2^{\mu - 2} \Gamma\left(\frac{\mu - \nu}{2}\right)
  \Gamma\left(\frac{\mu + \nu}{2}\right).
\]
Applying the Mellin inversion formula and using that
\[
  \pi^{-2s} \Gamma\left(\frac{s}{2}\right) \Gamma\left(\frac{s+1}{2}\right)^{2}
  \Gamma\left(\frac{s+2}{2}\right)
  = (2 \pi)^{1-2s} \Gamma(s) \Gamma(s+1)
\]
we obtain that
\[
  \phi(t) = 8 \pi^{2} \sqrt{t} K_{1}(4 \pi \sqrt{t}).
\]
Thus,
\[
  \Theta(t) = \sum_{n=1}^{\infty} 8 \pi^{2} b(n) \sqrt{\frac{nt}{N}}
  K_{1}\left(4 \pi \sqrt{\frac{nt}{N}}\right), \text{ and }
  \Theta'(t) = \sum_{n=1}^{\infty} 8 \pi^{2} b(n) \left(-\frac{2 \pi n}{N}
  K_{0}\left(4 \pi \sqrt{\frac{nt}{N}}\right)\right).
\]
Taking the two formulas above, rewriting $b(n)$ as a sum over $m$ and $d$
with $n = md^{2}$, and switching the order of summation gives the desired
formula.
\end{proof}

\section{Proof of Theorem~\ref{bound}}
\label{pfofbound}

In this section, we use the results from Section~\ref{RS} to prove
Theorem~\ref{bound}. Assume as in the previous section that $Q$ is a
positive-definite integer-valued quaternary quadratic form with
fundamental discriminant $D = D(Q)$ and Gram matrix $A$. In this case, the
level $N = N(Q)$ of $Q$ will equal $D$, and we will use $D$ and $N$
interchangeably in what follows.

Define the quadratic form $Q^{*}$ by $Q^{*}(\vec{x}) = \frac{1}{2} \vec{x}^{T} N A^{-1} \vec{x}$
and let
\begin{align*}
  \theta_{Q}(z) &= \sum_{n=0}^{\infty} r_{Q}(n) q^{n} = E(z) + C(z), \text{ and }\\
  \theta_{Q^{*}}(z) &= \sum_{n=0}^{\infty} r_{Q^{*}}(n) q^{n} = E^{*}(z) + C^{*}(z).
\end{align*}
Here $E(z), E^{*}(z)$ are the Eisenstein series and
$C(z), C^{*}(z) \in S_{2}(\Gamma_{0}(N), \chi)$. We cannot immediately apply
the formulas from Section~\ref{RS} to estimate $\langle C, C \rangle$ because
it is not generally true that $C(z) \in S_{2}^{\epsilon}(\Gamma_{0}(N), \chi)$
for $\epsilon = 1$ or $\epsilon = -1$. However, the following result allows us
to work with $C^{*}$ instead.

\begin{prop}
\label{cstareq}
We have $\langle C, C \rangle = N \langle C^{*}, C^{*} \rangle$.
Moreover, $C^{*} \in S_{2}^{-}(\Gamma_{0}(N), \chi)$.
\end{prop}
\begin{proof}
First, Proposition 10.1 of \cite{Iwa} (pg. 167) shows that
\[
  \theta_{Q} | W_{N} = -\sqrt{N} \theta_{Q^{*}}.
\]
The projection of $\theta_{Q}$ onto the space of Eisenstein
series (forms in $M_{2}(\Gamma_{0}(N), \chi)$ that are orthogonal to
all cusp forms) is $E(z)$. It follows that
$(-1/\sqrt{N}) E(z) | W_{N}$ is the projection of $\theta_{Q^{*}}$
onto the Eisenstein subspace and so $C | W_{N} = -\sqrt{N} C^{*}$.
Finally, $W_{N}$ is an isometry with respect to the Petersson inner product
(by Proposition 5.5.2 on page 185 of \cite{DS}). It follows that
\[
  \langle C, C \rangle = \langle C | W_{N}, C | W_{N} \rangle
  = N \langle C^{*}, C^{*} \rangle.
\]
This proves the first statement.

For the second statement, we will show that $\theta_{Q^{*}} \in
M_{2}^{-}(\Gamma_{0}(N), \chi)$. This implies that $E^{*} \in
M_{2}^{-}(\Gamma_{0}(N), \chi)$, since it is a linear combination of
the theta series in $\Gen(Q^{*})$, and this in turn implies that
$C^{*} \in S_{2}^{-}(\Gamma_{0}(N), \chi)$.

Proving that $\theta_{Q^{*}} \in M_{2}^{-}(\Gamma_{0}(N), \chi)$ is a
fun exercise using $\epsilon$-invariants. Factor the Dirichlet character
$\chi$ as
\[
  \chi = \prod_{p | 2N} \chi_{p},
\]
where for each prime $p$, $\chi_{p}$ is a primitive Dirichlet character
whose conductor is a power of $p$. Since $\cond(\chi) = N$,
we have that if $p > 2$, $\chi_{p}(m) = \legen{m}{p}$. We will show that if
$p$ is an odd prime dividing $N$, then $\epsilon_{p}(Q)$ equals
$\chi_{p}(m)$, where $m$ is any integer relatively prime to $N$ that
is represented by $Q^{*}$, while for $p = 2$,
$\epsilon_{2}(Q) = -\chi_{2}(m)$.

From the relation
\[
  \prod_{p | 2N} \epsilon_{p}(Q) = 1,
\]
we have that if $m$ is represented by $Q^{*}$ and $\gcd(m,N) = 1$, then
\[
  \chi(m) = \prod_{p | 2N} \chi_{p}(m) = -\epsilon_{2}(Q) \prod_{\substack{p | N\\ p > 2}} \epsilon_{p}(Q) = -1.
\]
This proves that $\theta_{Q^{*}} \in M_{2}^{-}(\Gamma_{0}(N), \chi)$.

Suppose that $p$ is an odd prime with $p | N$. Since $\chi$ is primitive,
it follows that $\ord_{p}(D) = \ord_{p}(N) = 1$. It follows that the
local Jordan splitting of $Q$ is one of the options listed in the table.

\begin{center}
\begin{tabular}{ccc}
Form & Determinant square-class & $\epsilon$\\
\hline
$x^{2} + y^{2} + z^{2} + pw^{2}$ & $p$ & $1$\\
$x^{2} + y^{2} + nz^{2} + npw^{2}$ & $p$ & $-1$\\
$x^{2} + y^{2} + z^{2} + npw^{2}$ & $np$ & $1$\\
$x^{2} + y^{2} + nz^{2} + pw^{2}$ & $np$ & $-1$
\end{tabular}
\end{center}
Here $n$ represents an element of $\Z_{p}^{\times}$ that is not a square.

If the local Jordan splitting of the form $Q$ is $ax^{2} + by^{2} + cz^{2} +
dw^{2}$, where $d$ is either $p$ or $np$, the local splitting of the form
$Q^{*}$ is $Na^{-1} x^{2} + Nb^{-1} y^{2} + Nc^{-1} z^{2} + Nd^{-1} w^{2}$.
It follows that if $m$ is represented by $Q^{*}$ and $m$ is coprime to
$p$, then $\chi_{p}(m) = \chi_{p}(Nd^{-1})$. If $N/p$ is a square mod $p$, then
the determinant square class of $Q$ is $p$.
It follows that $Nd^{-1}$ is a square mod $p$ if and only if $\epsilon = 1$.
If $N/p$ is not a square mod $p$, the determinant square class of $Q$ is
$np$ and once again $Nd^{-1}$ is a square mod $p$ if and only if $\epsilon = 1$.
This proves that $\chi_{p}(Nd^{-1}) = \epsilon_{p}(Q)$ if $p$ is odd.

Over $\Z_{2}$ every integral quadratic form can be decomposed as a sum of
diagonal terms, and blocks of the form
$\left[ \begin{matrix} 0 & 1 \\ 1 & 0 \end{matrix} \right]$, and
$\left[ \begin{matrix} 2 & 1 \\ 1 & 2 \end{matrix} \right]$ (see
\cite{Jones}). If the $D = \det A$ is odd, then its Jordan
splitting over $\Z_{2}$ cannot contain any diagonal
components. Therefore its splitting must consist of two blocks. In the
case that $D \equiv 1 \pmod{8}$, the two blocks must be
$\left[ \begin{matrix} 0 & 1 \\ 1 & 0 \end{matrix} \right]$, and in
the case when $D \equiv 5 \pmod{8}$, one block is
$\left[ \begin{matrix} 0 & 1 \\ 1 & 0 \end{matrix} \right]$ and the other
is $\left[ \begin{matrix} 2 & 1 \\ 1 & 2 \end{matrix} \right]$. Over
$\Q_{2}$, the blocks $\left[ \begin{matrix} 0 & 1 \\ 1 &
    0 \end{matrix} \right]$ and $\left[ \begin{matrix} 2 & 1 \\ 1 &
    2 \end{matrix} \right]$ are equivalent to $2x^{2} - 2y^{2}$ and
$2x^{2} + 6y^{2}$. This means that the local Jordan splitting of $A$
is either
\[
  \left[ \begin{matrix}
  0 & 1 & 0 & 0 \\
  1 & 0 & 0 & 0 \\
  0 & 0 & 0 & 1 \\
  0 & 0 & 1 & 0 \\ \end{matrix} \right], \text{ or }
  \left[ \begin{matrix}
  0 & 1 & 0 & 0 \\
  1 & 0 & 0 & 0 \\
  0 & 0 & 2 & 1 \\
  0 & 0 & 1 & 2 \end{matrix} \right].
\]
These are equivalent to $x^{2} - y^{2} + z^{2} - w^{2}$
and $x^{2} - y^{2} + z^{2} + 3w^{2}$ respectively, and both of these
have $\epsilon = -1$.

When the level is a multiple of 4 but not a multiple of $8$, one can
see that the quadratic form is equivalent over $\Z_{2}$ to either
\[
  \left[\begin{matrix}
   2a & 0 & 0 & 0\\
   0 & 2b & 0 & 0\\
   0 & 0 & 0 & 1\\
   0 & 0 & 1 & 0\end{matrix}\right], \text{ or }
  \left[\begin{matrix}
   2a & 0 & 0 & 0\\
   0 & 2b & 0 & 0\\
   0 & 0 & 2 & 1\\
   0 & 0 & 1 & 2 \end{matrix}\right]
\]
where $ab \equiv 1 \pmod{4}$. A straightforward calculation shows that
in this case $\epsilon \equiv a \pmod{4}$. The local splitting of $Q^{*}$
shows that the relevant part (mod 4) is $\frac{N}{4} ax^{2} + \frac{N}{4} by^{2}$. Since $N \equiv 0 \pmod{4}$ and $N/4 \equiv 3 \pmod{4}$, this shows that
the $2$-adic squareclass represented by $Q^{*}$ is $-\epsilon_{2}(Q)$.

When the level is a multiple of 8, the quadratic form is equivalent over
$\Z_{2}$ to either
\[
  \left[\begin{matrix}
    2a & 0 & 0 & 0\\
    0 & 4b & 0 & 0\\
    0 & 0 & 0 & 1\\
    0 & 0 & 1 & 0\\ \end{matrix} \right], \text{ or }
  \left[\begin{matrix}
    2a & 0 & 0 & 0\\
    0 & 4b & 0 & 0\\
    0 & 0  & 2 & 1\\
    0 & 0  & 1 & 2 \end{matrix}\right].
\]
The form $Q^{*}$ represents precisely two odd integers mod $8$:
$(D/8) b^{-1}$ and $(D/8) (b^{-1} + 2 a^{-1})$. A calculation of all 32
options and their $\epsilon$-invariants reveals that the desired result is
true in this case as well. This concludes the proof.
\end{proof}

In order to bound the largest locally represented integer not
represented by $Q$, we will require upper and lower bounds
on the Eisenstein series coefficients $a_{E}(n)$ and $a_{E^{*}}(n)$.

\begin{lem}
\label{eisensteinbounds}
For any $\epsilon > 0$, we have
\[
  \frac{n^{1 - \epsilon}}{N^{1/2}} \ll a_{E}(n) \ll \frac{n^{1 + \epsilon}}{N^{1/2}}
\]
if $n$ is locally represented by $Q$, and
\[
  \frac{n^{1 - \epsilon}}{N^{3/2}} \ll a_{E^{*}}(n) \ll \frac{n^{1 + \epsilon}}{N^{3/2-\epsilon}},
\]
if $n$ is locally represented by $Q^{*}$. The implied constants depend only
on $\epsilon$.
\end{lem}

\begin{rem}
We use the above Lemma only for the proof of Theorem~\ref{bound}. For
the proof of Theorem~\ref{451}, we use computer calculations with local
densities to derive completely explicit bounds that depend on the form $Q$.
\end{rem}

\begin{proof}
We have the formula
\[
  a_{E}(n) = \prod_{p \leq \infty} \beta_{p}(n).
\]
In \cite{Yang},
formulas are given for the local densities $\beta_{p}(n)$ (in Yang's notation,
these are $\alpha_{p}(n,\frac{1}{2} A)$). See in particular
Theorem 3.1 for $p > 2$ and Theorem 4.1 for $p = 2$. We have
$\beta_{p}(n) = 1$ if $p > 2$ and $p \nmid n$.

If $p$ is odd and $p \nmid N$, then Theorem 3.1 of \cite{Yang} gives
the bounds $1 - \frac{1}{p} \leq \beta_{p}(n) \leq 1 + \frac{1}{p}$.
If $p$ is odd and $p | N$, we get the same bound for the form $Q$.
For the form $Q^{*}$ we get $1 - \frac{1}{p} \leq \beta_{p}(n) \leq 2$
provided $n$ is locally represented. Theorem 4.1 of \cite{Yang} shows
that there is an absolute upper bound on $\beta_{2}(n)$ over all positive
integers $n$ and all forms $Q$ and $Q^{*}$ with discriminants $N$ and
$N^{3}$, where $N$ is a fundamental discriminant.

Notice that neither $Q$ nor $Q^{*}$ can be
anisotropic at any prime. There is a unique $\Q_{p}$-equivalence class
of quaternary quadratic forms that is anisotropic at $p$,
and such forms must have discriminant a square. The discriminant of
$Q$ is $N$ and the discriminant of $Q^{*}$ is $N^{3}$, and neither of
these are squares in $\Q_{p}$ if $p | N$. From this and the recursion formulas
of Hanke \cite{Hanke} it follows that there is an absolute lower bound
on $\beta_{2}(n)$ over all quaternary forms $Q$ with fundamental discriminant
that locally represent $n$, and similarly for $Q^{*}$.
Finally, $\beta_{\infty}(n) = \frac{\pi^{2} n}{\sqrt{D}}$.

Putting these bounds together gives
\begin{align*}
  \frac{n}{\sqrt{N}} \prod_{p | n} \left(1 - \frac{1}{p}\right)
  \ll & a_{E}(n) \ll \frac{n}{\sqrt{N}} \prod_{p | n} \left(1 + \frac{1}{p}\right)\\
  \frac{n^{1 - \epsilon}}{N^{1/2}} \ll & a_{E}(n) \ll \frac{n^{1 + \epsilon}}{N^{1/2}}.
\end{align*}

For $Q^{*}$ we have
\begin{align*}
  \frac{n}{\sqrt{N^{3}}} \prod_{p | n} \left(1 - \frac{1}{p}\right)
  \ll & a_{E^{*}}(n) \ll \frac{n}{\sqrt{N^{3}}} \prod_{p | n}
  \left(1 + \frac{1}{p}\right) \prod_{p | N} 2\\
  \frac{n^{1-\epsilon}}{N^{3/2}} \ll
 & a_{E^{*}}(n) \ll \frac{n^{1 + \epsilon}}{N^{3/2 - \epsilon}},
\end{align*}
since $\prod_{p | N} 2 \leq d(N) \ll N^{\epsilon}$.
\end{proof}

Prior to stating and proving our bound on $\langle C, C \rangle$ we need
a few more preliminary observations. The first is related to bounding the sum
\[
  \sum_{d=1}^{\infty} \psi\left(d \sqrt{\frac{n}{N}}\right).
\]
Since
\[
  \psi(x) = -\frac{6}{\pi} x K_{1}(4 \pi x) + 24 x^{2} K_{0}(4 \pi x),
\]
and $K_{1}(x)$ is positive, it follows that $\psi(x) \leq 24 x^{2} K_{0}(4 \pi x)$.
Using formula (10.32.9) of \cite{DLMF}, we have the bound
\begin{equation}
\label{KBesselbound}
  K_{0}(x) = \int_{0}^{\infty} e^{-x \cosh(t)} \, dt
  \leq \int_{0}^{\infty} e^{-x (1 + t^{2}/2)} \, dt = \sqrt{\frac{\pi}{2 x}} e^{-x}.
\end{equation}
It follows that $\psi(x)$ is decreasing exponentially, and hence
$\sum_{d=1}^{\infty} \psi(dx)$ is bounded if $x \gg 1$. In addition,
\[
\hat{\psi}(y) = -\frac{9y^{2}}{\pi^{2} (4 + y^{2})^{5/2}}
\]
and the Poisson summation formula implies that
\[
  -\frac{3}{2 \pi^{2}} + 2 \sum_{d=1}^{\infty} \psi(dx) =
  2 \sum_{d=1}^{\infty} \frac{1}{x} \hat{\psi}(d/x),
\]
which shows that $\sum_{d=1}^{\infty} \psi(dx) \to \frac{3}{4 \pi^{2}}$
as $x \to 0$.

Our next lemma is a bound on $\sum_{n \leq x} d(n) r_{Q^{*}}(n)^{2}$
which will be useful in bounding $\langle C, C \rangle$.

\begin{lem}
\label{sumbound}
Assume the notation above. We have
\[
  \sum_{n \leq x} d(n) r_{Q^{*}}(n)^{2} 
  \ll_{\epsilon} \begin{cases}
  x^{1/2 + \epsilon} & \text{ if } x \leq N^{1/2},\\  
  \frac{x^{1 + \epsilon}}{N^{1/4}} & \text{ if } N^{1/2} \leq x \leq N^{5/6},\\
  \frac{x^{3/2 + \epsilon}}{N^{2/3}} & \text{ if } N^{5/6} \leq x \leq N^{11/12},\\
  \frac{x^{2 + \epsilon}}{N^{9/8}} & \text{ if } N^{11/12} \leq x \leq N,\\
  \frac{x^{7/2 + \epsilon}}{N^{21/8}} & \text{ if } x \geq N.
  \end{cases}
\]
Moreover, for $n \geq N^{11/12}$, we have $r_{Q^{*}}(n) \leq \frac{n^{3/2}}{N^{9/8}}$.
\end{lem}
\begin{proof}
We use that $d(n) \ll n^{\epsilon}$ to get
\[
  \sum_{n \leq x} d(n) r_{Q^{*}}(n)^{2} \ll
  x^{\epsilon} \sum_{n \leq x} r_{Q^{*}}(n)^{2}
  \ll x^{\epsilon} \left(\sum_{n \leq x} r_{Q^{*}}(n)\right) \cdot
  \left(\max_{n \leq x} r_{Q^{*}}(n)\right).
\]
First, we will bound $\sum_{n \leq x} r_{Q^{*}}(n)$. Theorem 2.1.1 of Kitaoka's book \cite{Kitaoka} shows that we may write
the Gram matrix of $Q$ as
\[
  A = M^{T} D M,
\]
where $M$ is an upper triangular matrix with ones on the diagonal,
and $D$ is a diagonal matrix with entries $a_{1}$, $a_{2}$, $a_{3}$, and $a_{4}$
where $a_{i}/a_{i+1} \leq 4/3$ for $i \geq 1$ and $a_{1} \geq 1$. This implies
that $a_{2} \geq 3/4$, $a_{3} \geq 9/16$ and $a_{4} \geq 27/64$. Since
$a_{1} a_{2} a_{3} a_{4} = N$, it follows that $a_{i} \ll N$ for all $i$.

Taking the inverse and multiplying by $N$ gives
\[
  A^{*} = N A^{-1} = M^{-1} (N D^{-1}) (M^{-1})^{T}.
\]
If we let $a_{i}^{*} = N/a_{i}$, then we have written
\[
  Q^{*}(x_{1}, x_{2}, x_{3}, x_{4}) =
  a_{1}^{*} (x_{1} + m_{12} x_{2} + m_{13} x_{3} + m_{14} x_{4})^{2}
  + a_{2}^{*} (x_{2} + m_{23} x_{3} + m_{24} x_{4})^{2} +
  a_{3}^{*} (x_{3} + m_{34} x_{4})^{2} + a_{4}^{*} x_{4}^{2}.
\]
We have that $a_{i}^{*} \ll N$, $a_{i}^{*} \gg
1$, and $a_{1}^{*} a_{2}^{*} a_{3}^{*} a_{4}^{*} = N^{3}$. From
the centered equation above, it follows that if $Q^{*}(x_{1}, x_{2}, x_{3}, x_{4})
\leq x$, then $x_{i}$ is in an interval of length at most
$2 \sqrt\frac{x}{a_{i}^{*}}$. Thus,
\[
  \sum_{n \leq x} r_{Q^{*}}(n) \leq
  \prod_{i=1}^{4} \left(2 \sqrt{\frac{x}{a_{i}^{*}}} + 1\right).
\]
Since $N^{3} = a_{1}^{*} a_{2}^{*} a_{3}^{*} a_{4}^{*}$, we have that
$a_{i} a_{j} \gg N$ and $a_{i} a_{j} a_{k} \gg N^{2}$. Expanding the product
on the right hand side gives that
\begin{align*}
  \sum_{n \leq x} r_{Q^{*}}(n) \ll
  \begin{cases}
    \sqrt{x} & x \leq N\\
    \frac{x^{2}}{N^{3/2}} & x \geq N.
\end{cases}
\end{align*}

A very similar argument gives a bound on $r_{Q^{*}}(n)$. Assume without
loss of generality that $a_{1}^{*} \geq a_{2}^{*} \geq a_{3}^{*} \geq a_{4}^{*}$.
In order for $Q^{*}(x_{1}, x_{2}, x_{3}, x_{4})$ to be equal to $n$, we can allow 
$x_{i}$, $1 \leq i \leq 3$ to range over an interval (depending on the
values of the other $x_{i}$) 
of length $2 \sqrt{n/a_{i}^{*}}$. Given the choices of $x_{1}$, $x_{2}$ and
$x_{3}$, the formula $Q^{*}(x_{1}, x_{2}, x_{3}, x_{4}) = n$ is a quadratic
equation in $x_{4}$ and has at most two solutions. This proves that
\[
  r_{Q^{*}}(n)
  \leq 2 \prod_{i=1}^{3} \left(2 \sqrt{\frac{n}{a_{i}^{*}}} + 1\right).
\]
Choosing $n \leq x$ and expanding the product gives
\[
  \max_{n \leq x} r_{Q^{*}}(n)
  \ll \frac{x^{3/2}}{\sqrt{a_{1}^{*} a_{2}^{*} a_{3}^{*}}}
  + \frac{x}{\sqrt{a_{2}^{*} a_{3}^{*}}} + \frac{\sqrt{x}}{\sqrt{a_{3}^{*}}} + 1.
\]
The bounds on the $a_{i}$ imply that $\frac{1}{\sqrt{a_{1}^{*} a_{2}^{*} a_{3}^{*}}}
\ll \frac{1}{N^{9/8}}$, $\frac{1}{\sqrt{a_{2}^{*} a_{3}^{*}}} \ll \frac{1}{N^{2/3}}$, and $\frac{1}{\sqrt{a_{3}^{*}}} \ll \frac{1}{N^{1/4}}$. This yields
\begin{equation}
\label{maxbound}
  \max_{n \leq x} r_{Q^{*}}(n) \ll
  \begin{cases}
    1 & \text{ if } x \leq N^{1/2}\\
    \frac{x^{1/2}}{N^{1/4}} & \text{ if } N^{1/2} \leq x \leq N^{5/6}\\
    \frac{x}{N^{2/3}} & \text{ if } N^{5/6} \leq x \leq N^{11/12}\\
    \frac{x^{3/2}}{N^{9/8}} & \text{ if } x \geq N^{11/12}.
\end{cases}
\end{equation}
This yields the second stated result. Combining \eqref{maxbound} with $d(n) \ll
x^{\epsilon}$ for $n \leq x$ and $\sum_{n \leq x} r_{Q^{*}}(n) \ll
\max\left(\sqrt{x}, \frac{x^{2}}{N^{3/2}}\right)$, yields the first
stated result.
\end{proof}

Now, we bound the Petersson norm of $C(z)$, the cuspidal part of
$\theta_{Q}(z)$. This result is a significant improvement over
the result of Schulze-Pillot in \cite{SP}, where it is proven that 
$\langle C, C \rangle \ll N$, assuming that $N$ is square-free. This 
improvement has two sources: (i) the formula from Proposition~\ref{petform} 
has a factor of $n$ in the denominator of the $n$th term, and (ii) 
Schulze-Pillot uses a bound of the shape 
$r_{Q^{*}}(n) \ll \frac{n^{2}}{\sqrt{D}}$, which is much weaker than
the result of Lemma~\ref{sumbound}.

\begin{thm}
\label{petbound}
We have
\[
  \langle C, C \rangle \ll \frac{N}{\sigma(N)},
\]
where $\sigma(N)$ is the sum of the divisors of $N$.
\end{thm}
\begin{rem}
  It follows that $\langle C, C \rangle$ is bounded. Theorem~\ref{petbound} is
  sharp in the case that $\theta_{Q^{*}}$ represents an integer $n$
  bounded independently of $N$. In this case, $r_{Q^{*}}(n) \geq 1$,
  $a_{E^{*}}(n) \ll \frac{1}{N^{3/2 - \epsilon}}$ and so $a_{C^{*}}(n)
  \gg 1$. The proof then shows that $\langle C, C \rangle \gg \frac{N}{\sigma(N)}$.
\end{rem}

\begin{proof}
By Proposition~\ref{cstareq}, we have
$\langle C, C \rangle =
N \langle C^{*}, C^{*} \rangle$. Proposition~\ref{petform} then implies that
\[
  \langle C, C \rangle = \frac{N}{[\SL_{2}(\Z) : \Gamma_{0}(N)]}
  \sum_{n=1}^{\infty} \frac{2^{\omega(\gcd(n,N))} a_{C^{*}}(n)^{2}}{n}
  \sum_{d=1}^{\infty} \psi\left(d \sqrt{\frac{n}{N}}\right).
\]
Since $N$ is squarefree, $[\SL_{2}(\Z) : \Gamma_{0}(N)] = \sigma(N)$.
Since $K_{0}(x) \leq \sqrt{\frac{\pi}{2 x}} e^{-x}$, we have for $x \gg 1$ the estimate
\[
  \sum_{d=1}^{\infty} \psi(d \sqrt{x}) \ll \sum_{d=1}^{\infty}
  d^{3/2} x^{3/4} e^{-4 \pi d \sqrt{x}} \ll x^{3/4} e^{-4 \pi \sqrt{x}}.
\]
We have
$r_{Q^{*}}(n) = a_{E^{*}}(n) + a_{C^{*}}(n)$ and so
$|a_{C^{*}}(n)| \leq a_{E^{*}}(n) + r_{Q^{*}}(n)$. Observe that
$a_{C^{*}}(n)^{2} \leq 2 a_{E^{*}}(n)^{2} + 2 r_{Q^{*}}(n)^{2}$. 
We will first handle the terms $r_{Q^{*}}(n)^{2}$. Observing that
$2^{\omega(\gcd(n,N))} \leq d(n)$, these terms are bounded by
\begin{equation}
\label{term1}
  \frac{N}{\sigma(N)} \sum_{n=1}^{\infty}
  \frac{r_{Q^{*}}(n)^{2} d(n)}{n} \sum_{d=1}^{\infty}
  \psi\left(d \sqrt{\frac{n}{N}}\right).
\end{equation}

We first consider the tail. For $n \geq N$,
we have $r_{Q^{*}}(n) \ll \frac{n^{3/2}}{N^{9/8}}$ and $d(n) \ll n^{1/8}$. This
gives the bound
\[
  \sum_{n=k}^{\infty} \frac{n^{1/8} \cdot (n^{3}/N^{9/4})}{n}
  \cdot \left(\frac{n}{N}\right)^{3/4} e^{-4 \pi \sqrt{n/N}}\\
  = \frac{1}{N^{1/8}} \sum_{n=k}^{\infty}
  \left(\frac{n}{N}\right)^{23/8} e^{-4 \pi \sqrt{n/N}}.
\]
Breaking the sum into the pieces $r^{2} N \leq n \leq (r+1)^{2} N$
gives the bound
\[
  \frac{1}{N^{1/8}} \sum_{r=\lfloor \sqrt{\frac{k}{N}} \rfloor}^{\infty}
  (2r+1) N (r+1)^{23/8} e^{-4 \pi r}
  \ll N^{7/8} \sum_{r=\lfloor \sqrt{\frac{k}{N}} \rfloor}^{\infty}
  r^{4} e^{-4 \pi r}.
\]
Let $f(r) = r^{4} e^{-4 \pi r}$ and note that for $r \geq 1$, $f(r+1) \leq
(1/2) f(r)$. The sum is therefore bounded by
$2 N^{7/8} f(\lfloor \sqrt{\frac{k}{N}} \rfloor)$. 
We choose $k$ so that 
$\log(N) \leq \lfloor \sqrt{\frac{k}{N}} \rfloor \leq \log(N) + 1$.
The sum is then $\ll N^{7/8} \log(N)^{4} / N^{4 \pi} = O(N^{-11})$
and $k \ll N \log^{2}(N)$.

Now, we handle the terms with $n \ll N \log^{2}(N)$. Recall
that $\sum_{d=1}^{\infty} \psi\left(d \sqrt{\frac{n}{N}}\right)$ is bounded.
Thus, we estimate
\[
  \sum_{n=1}^{cN \log^{2}(N)} \frac{d(n) r_{Q^{*}}(n)^{2}}{n}
  = \int_{1}^{\infty}
  \frac{1}{t^{2}} \left(\sum_{n \leq \min(t, cN \log^{2}(N))} d(n) r_{Q^{*}}(n)^{2}\right) \, dt.
\]
We use Lemma~\ref{sumbound} repeatedly. In the range $1 \leq t \leq \sqrt{N}$,
we get $\int_{1}^{\sqrt{N}} \frac{t^{1/2 + \epsilon}}{t^{2}} \, dt$ which is bounded.
Indeed this is the main contribution to $\langle C, C \rangle$.

The other ranges yield
\begin{align*}
  &\int_{\sqrt{N}}^{N^{5/6}} \frac{t^{1 + \epsilon}}{t^{2} N^{1/4}} \, dt
  + \int_{N^{5/6}}^{N^{11/12}} \frac{1}{t^{1/2 - \epsilon} N^{2/3}} \, dt
  + \int_{N^{11/12}}^{N} \frac{t^{\epsilon}}{N^{9/8}} \, dt
  + \int_{N}^{c N \log^{2}(N)} \frac{t^{3/2 + \epsilon}}{N^{21/8}} \, dt\\
  &+ \int_{cN \log^{2}(N)}^{\infty} \frac{(N \log^{2}(N)^{7/2 + \epsilon}}{t^{2} N^{21/8}} \, dt = O(N^{-1/4 + 5\epsilon/6}) + O(N^{-5/24 + 11/12\epsilon})
  + O(N^{-1/8 + \epsilon})\\ 
  &+ O(N^{-1/8 + \epsilon} \log^{5 + 2\epsilon}(N)) + O(N^{-1/8} \log^{5 + 2\epsilon}(N)).
\end{align*}
This shows that $\sum_{n=1}^{cN \log^{2}(N)} \frac{d(n) r_{Q^{*}}(n)^{2}}{n} \ll 1$.

Equation~\eqref{genusaverage} shows that the Eisenstein
series $E^{*}(z) = \sum_{i=1}^{m} c_{i} \theta_{R_{i}}(z)$, where the forms
$R_{i}$ are the forms in the genus of $Q^{*}$ and $\sum_{i=1}^{m} c_{i} = 1$.
It is easy to see that $\sum_{n \leq x} d(n) a_{E^{*}}(n)^{2}$ obeys exactly
the same bound as $\sum_{n \leq x} d(n) r_{Q^{*}}(n)^{2}$ by applying
Lemma~\ref{sumbound} to bound $\sum_{n \leq x} r_{R_{i}}(n)$ and 
$\max_{n \leq x} r_{R_{i}}(n)$. The contribution from
the terms involving $a_{E^{*}}(n)^{2}$ is therefore also bounded.
Hence,
\[
  \sum_{n=1}^{\infty} \frac{2^{\omega(\gcd(n,N))} a_{C^{*}}(n)^{2}}{n}
  \sum_{d=1}^{\infty} \psi\left(d \sqrt{\frac{n}{N}}\right)
  \ll 1
\]
and this gives the overall bound of $\langle C, C \rangle \ll \frac{N}{\sigma(N)}$, as desired.
\end{proof}

Finally, we are ready to prove Theorem~\ref{bound}.

\begin{proof}[Proof of Theorem~\ref{bound}]
Fix $\epsilon > 0$. Write
\[
  \theta_{Q}(z) = \sum_{n=0}^{\infty} r_{Q}(n) q^{n}
  = E(z) + C(z),
\]
where $E(z) = \sum_{n=0}^{\infty} a_{E}(n) q^{n}$ and
$C(z) = \sum_{n=0}^{\infty} a_{C}(n) q^{n}$. We have
\[
  |a_{C}(n)| \leq C_{Q}^{{\rm odd}} d(n) \sqrt{n}
\]
where $C_{Q} \leq \sqrt{\frac{\langle C, C \rangle u}{B}}$ by
\eqref{cqdef}. Here $u = \dim S_{2}(\Gamma_{0}(N), \chi)$ and $B$ is a lower
bound for the Petersson norm of a newform in $S_{2}(\Gamma_{0}(N), \chi)$.
It follows from the work of Hoffstein and Lockhart \cite{GHL} that
$B \gg N^{-\epsilon}$, although this bound is ineffective. From
Theorem~\ref{petbound}, we have $\langle C, C \rangle \ll N^{\epsilon}$.
Combining this with $u \ll N$ gives that $C_{Q} \ll N^{1/2 + \epsilon/2}$.

Now, from Lemma~\ref{eisensteinbounds}, we have
$a_{E}(n) \gg \frac{n^{1 - \epsilon/2}}{\sqrt{N}}$ provided $n$ is locally
represented by $Q$. Combining these estimates, we have that $r_{Q}(n)$
is positive if $n$ is locally represented by $Q$ and
\[
  \frac{n^{1 - \epsilon/2}}{\sqrt{N}} \gg N^{1/2 + \epsilon/2} d(n) \sqrt{n}.
\]
Since $d(n) \ll n^{\epsilon/2}$, any locally represented $n$ satisfying
$n \gg N^{2 + \epsilon}$ is represented.

If $f$ is a newform, then $f | W_{N}$ is also a newform. It follows
from this fact and from Proposition~\ref{cstareq} that $C_{Q^{*}}
= \frac{1}{\sqrt{N}} C_{Q}$. Therefore $r_{Q^{*}}(n)$ is positive
if $n$ is locally represented by $Q^{*}$ and
\[
  \frac{n^{1 - \epsilon/4}}{N^{3/2 - \epsilon/2}}
  \gg n^{1/2 + \epsilon/4}.
\]
This implies that $n^{1/2 - \epsilon/2} \gg N^{3/2 - \epsilon/2}$, which
yields $n \gg N^{3 + \epsilon}$.
\end{proof}

\section{Proof of the 451-Theorem}
\label{pfof451}

If $L$ is a lattice, we say that $L$ is \emph{odd universal} if every
odd positive integer is the norm of a vector $\vec{x} \in L$. Such lattices
(up to isometry) are in bijection with positive-definite integer-valued quadratic
forms $Q$ (up to equivalence) that represent all positive odd integers.

We use the approach (and terminology) pioneered by Bhargava
\cite{Bhar} and used in \cite{BH} to prove the 290-Theorem. An
\emph{exception} for a lattice $L$ is an odd positive integer that
does not occur as the norm of a vector in $L$. If $L$ is a lattice
that is not odd universal, we define the \emph{truant} of $L$ to be
the smallest positive odd integer $t$ that is the not the norm of
a vector in $L$. An \emph{escalation} of $L$ is a lattice $L'$
generated by $L$ and a vector of norm $t$. We will study the
escalations of the dimension zero lattice, and call all such lattices
generated by this process \emph{escalator lattices}. Finally,
the $46$ odd integers given in the statement of the 451-Theorem
are called the \emph{critical integers}.

Note that if $L$ is an odd universal lattice, then there is a sequence
of escalator lattices
\[
  \{ 0 \} = L_{0} \subseteq L_{1} \subseteq L_{2} \subseteq
  \cdots \subseteq L_{n} \subseteq L
\]
where $L_{i+1}$ is an escalation of $L_{i}$ for $0 \leq i \leq n-1$,
and $L_{n}$ is odd universal.

We begin by escalating the zero-dimensional lattice by a vector of norm $1$,
and getting the unique one-dimensional escalation with Gram matrix $\left[
\begin{matrix} 2 \end{matrix} \right]$
and quadratic form $x^{2}$. This lattice has truant $3$
and its escalations have Gram matrices of the form
$\left[ \begin{matrix}
  2 & a\\ a & 6 \end{matrix} \right]$.
We have $a = 2 \langle \vec{x}, \vec{y} \rangle$ where $\vec{x}$ and
$\vec{y}$ are vectors of norms $1$ and $3$. By the Cauchy-Schwarz
inequality, we have $|a| \leq 2 \sqrt{3}$. Up to isometry, we get four Gram
matrices:
\[
  \left[ \begin{matrix}
    2 & 0 \\ 0 & 6 \end{matrix} \right],
  \left[ \begin{matrix}
    2 & 1 \\ 1 & 6 \end{matrix} \right],
  \left[ \begin{matrix}
    2 & 0 \\ 0 & 4 \end{matrix} \right], \text{ and }
  \left[ \begin{matrix}
    2 & 1 \\ 1 & 2 \end{matrix} \right].
\]
These lattices have truants $5$, $7$, $5$, and $5$ respectively. Escalating
these four two-dimensional lattices gives rise to $73$ three dimensional
lattices. Twenty-three of these lattices correspond
to the $23$ ternary quadratic forms given in \cite{Kap}.
Conjecture~\ref{Kaplansky} states that these represent all positive odds,
and we assume Conjecture~\ref{Kaplansky} for the rest of this section.

Escalating the $50$ ternary lattices that are not odd universal gives
rise to the $24312$ \emph{basic} four-dimensional escalators. Of the $24312$,
$23513$ represent every positive odd integer less than
$10000$. Of the remaining $799$, $795$ locally represent all odd
numbers, and hence represent all but finitely many squarefree odds.
The remaining four fail to locally represent all odd integers:
\begin{align*}
  & x^{2} + 3y^{2} + 5z^{2} + 7w^{2} - 3yw,\\
  & x^{2} + 3y^{2} + 5z^{2} + 6w^{2} - xw - 2yw + 5zw,\\
  & x^{2} + 3y^{2} + 5z^{2} + 11w^{2} - xw - 2yw, \text{ and }\\
  & x^{2} + 3y^{2} + 7z^{2} + 9w^{2} + xy - xw.
\end{align*}
The first three fail to represent integers of the form $5n$, where $n
\equiv 3 \text{ or } 7 \pmod{10}$ and the last fails to represent
integers of the form $7n$ for $n \equiv 3, 5, 13 \pmod{14}$. To handle
these four forms, we compute \emph{auxiliary} escalator lattices.  The
first three lattices have truant $15$, and the fourth has truant $21$.
The auxiliary escalator lattices are those new lattices obtained by
escalating $x^{2} + 3y^{2} + 5z^{2}$ by $15$ (there are $196$) and
$x^{2} + xy + 3y^{2} + 7z^{2}$ by $21$ (there are $384$). All of these
auxiliary lattices locally represent all odds, and every odd universal
lattice contains a sublattice isometric to one of the 23 odd universal
ternaries, or one of the $24888 = 24312 + 196 + 384 - 4$
four-dimensional escalators (basic or auxiliary). It follows from this
that there are only finitely many escalator lattices. We now seek to
determine precisely which squarefree positive odd integers are
represented by each of these $24888$ quadratic forms. When we refer to
a form by number, it refers to the index of the form on the list of
the $24888$ in the file {\tt quatver.txt} (available at {\tt
  http://www.wfu.edu/$\sim$rouseja/451/}).

{\bf Method 1}: Universal ternary sublattices.

If $L$ is a quaternary lattice with a sublattice $L'$ that is one
of the $23$ odd universal lattices of dimension $3$, then the quadratic
form corresponding to $L$ represents all odd integers. Given $L$ it is
straightforward to check if such a lattice $L'$ exists, as it must be spanned
by vectors of norm $1$, $3$, $5$ and/or $7$. If such a lattice
exists, then the quadratic form corresponding to $L$ represents
all positive odds. The method applies to $2342$ of the $24888$,
and proves that each of these are odd universal.

\begin{ex}
Form $16451$ has level $2072$ and is the form with largest level to which
this method applies. It is given by
\[
  Q(x,y,z,w) = x^{2} + xy + xw + 3y^{2} + 7z^{2} + 7w^{2}.
\]
We have $Q(x,y,0,-z) = x^{2} + xy - xz + 3y^{2} + z^{2}$, which is
one of the forms given by Kaplansky in \cite{Kap}. The ternary
form $x^{2} + xy - xz + 3y^{2} + z^{2}$ has genus of size
1 and represents all positive odds. Hence
$Q$ represents all positive odds.
\end{ex}

{\bf Method 2}: Nicely embedded regular ternary sublattices.

Recall that a positive-definite quadratic form $Q$ is called
\emph{regular} if every locally represented integer $m$ is represented
by $Q$. In \cite{JKS}, Jagy, Kaplansky and Schiemann give a list of
$913$ ternary quadratic forms. They prove that every regular ternary
quadratic form appears on this list, and that $891$ of the forms on
this list are in fact regular. Of these, $792$ are in a genus of size
$1$, and are hence automatically regular (since the Hasse-Minkowski
theorem implies that a number that is locally represented is
represented by some form in the genus). This paper unfortunately does
not contain proofs of regularity for the $99$ forms but supplementary
documentation is available from Jagy upon request that supplies the
necessary proofs.  Recently, Bweong-Kweon Oh proved \cite{BOh} that
$8$ of the remaining $22$ conjecturally regular ternaries are in fact
regular.

We say that a quaternary lattice $L$ has a nicely embedded regular
ternary if there is a ternary sublattice $K$ whose corresponding quadratic
forms is regular, with the property that the quadratic form corresponding
to $K \oplus K^{\perp}$ locally represents all positive odds. We may
write the quadratic form corresponding to $K \oplus K^{\perp}$ as
\[
  T(x,y,z) + dw^{2}
\]
where $T(x,y,z)$ is one of the $891$ regular ternaries. Our approach
for determining the square-free odd numbers not represented by
$T(x,y,z) + dw^{2}$ is then as follows. If $n$ is an odd number, find
a representation for $n$ in the form
\[
  T(x,y,z) + dw^{2} = n.
\]
Since $T$ is regular, there is then a residue class $a + b \Z$ containing
$n$ so that $T(x,y,z) + dw^{2}$ represents all integers in $a + b \Z$
greater than or equal to $n$.

To determine the odd integers represented by a quaternary with a
nicely embedded regular ternary, we first find a modulus $M$ divisible by
all the primes dividing the discriminant of $T$ so that for each $a
\in \Z$ with $\gcd(a,M) = 1$ either $T$ locally represents everything
in the residue class $a \pmod{M}$ or $T$ does not locally represent
any integer in the residue class $a \pmod{M}$.

We then create a queue of residue classes to check, initially
containing all $a \pmod{M}$ that $T$ does not locally represent. Within
each residue class, we check each number to see if it is represented. If
a number is represented with $T(x,y,z) \ne 0$, one can find a residue
class $M' \geq M$ so that any number in the residue class
$a \pmod{M'}$ is represented. If $M' = M$, we are finished with
this residue class. If $M' > M$,
the residue classes $a + kM \pmod{M'}$ with $k \ne 0$ that contain
squarefree integers are added to the queue. When all of the residue classes
have been checked, we are left with a list of odd numbers not represented
by $K \oplus K^{\perp}$. It is then necessary to check to see if
$Q$ represents these numbers.

\begin{ex}
  If $Q = x^{2} + y^{2} + yz + 2z^{2} + 7w^{2}$, then $T = x^{2} +
  y^{2} + yz + 2z^{2}$ is a nicely embedded regular ternary. The form
  $T$ represents all positive integers except those of the form $n
  \equiv 21, 35, 42 \pmod{49}$.  We have
\[
  21 = 7 \cdot 1^{2} + 14, \quad
  35 = 7 \cdot 2^{2} + 7, \quad,
  42 = 7 \cdot 2^{2} + 14,
\]
and since $T$ represents every positive integer
$\equiv 7 \text{ or } 14 \pmod{49}$, $Q$ represents all positive integers.
\end{ex}

This method applies to $7470$ of the quaternaries. Many of these quaternaries
are escalations of the regular ternary form $x^{2} + xy + 3y^{2} + 4z^{2}$
with truant $77$, and some of these escalations have very large level.
For example, form $16367$
\[
  Q(x,y,z,w) = x^{2} + xy + 3y^{2} + 4z^{2} + zw + 77w^{2}
\]
has level $13541$, the largest of any of the $24888$, and form $16350$
\[
  Q(x,y,z,w) = x^{2} + xy + xw + 3y^{2} + 2yw + 4z^{2} - 2zw + 74w^{2}
\]
has $\theta_{Q} \in M_{2}(\Gamma_{0}(12900), \chi_{129})$ and
$\dim S_{2}(\Gamma_{0}(12900), \chi_{129}) = 2604$ (the largest
dimension of $S_{2}(\Gamma_{0}(N), \chi)$ for any of the $24888$). These forms
would be very unpleasant to deal with using other methods. Even though it is
occasionally necessary to check a large number of residue classes (as many as
$142081$), this method is quite efficient. None of the $7470$ quaternaries
tested using this method require more than $30$ minutes of computation time,
and much of this computation time is devoted to checking if $Q$ represents
numbers that are not represented by $K \oplus K^{\perp}$.

{\bf Method 3}: Rankin-Selberg $L$-functions.

We apply this method for the $8733$ quaternaries with fundamental discriminant
to which methods 1 and 2 do not apply. We use all of the machinery developed
in Section~\ref{RS}, although some modifications are desirable.

Suppose that $Q$ is a positive-definite, integer-valued quadratic form
with fundamental discriminant and level $N$. We use the following
procedure to determine which squarefree integers $Q$
represents. First, we compute a lower bound on $\langle g, g \rangle$
for all non-CM newforms in $S_{2}(\Gamma_{0}(N), \chi)$ using
Proposition~\ref{llow} (using the optimal choice of the parameter
given in equation \eqref{optimal}). Since the lower bound given on
$\langle g, g \rangle$ in the proof of Theorem~\ref{bound} is ineffective,
it is necessary to explicitly enumerate
the CM forms in $S_{2}(\Gamma_{0}(N), \chi)$ and estimate from below
their Petersson norms. We do this by finding all negative fundamental
discriminants $\Delta$ that divide $N$ and all ideals of norm
$|N|/|\Delta|$ in the ring of integers of the field
$\Q(\sqrt{\Delta})$. All Hecke characters with these moduli are
constructed, and then Magma's built-in routines for computing with Hecke
Gr\"ossencharacters are used to construct the CM forms $g$. Once this is done,
we compute enough terms of the Fourier expansion of $g$ so that the lower bound
we get on $\langle g, g \rangle$ from Proposition~\ref{petform} is at least
as large as our bound on $\langle g, g \rangle$ for non-CM $g$.

We then compute the first $15N$ coefficients of $\theta_{Q^{*}}$. We
pre-compute the local densities associated to $Q^{*}$ and use these to compute
the first $15N$ coefficient of $E^{*}$, and from this obtain
$C^{*} = \theta_{Q^{*}} - E^{*}$.  This data is plugged into
Proposition~\ref{petform}. The parts of this formula with $nd^{2} \leq 15N$
are explicitly computed. We bound the contribution from terms
with $n \leq 15N$ and $nd^{2} > 15N$ by using \eqref{KBesselbound},
giving that
\begin{align*}
  & \sum_{d > \sqrt{15N/n}} \psi\left(d \sqrt{\frac{n}{N}}\right)
  \leq \frac{6 \sqrt{2} n^{5/4}}{\sqrt{15N} N^{3/4}}
  \sum_{d = \lfloor \sqrt{\frac{15N}{m}} + 1 \rfloor}^{\infty}
  d^{2} e^{-4 \pi d \sqrt{n/N}}\\
  &= \frac{6 \sqrt{2} n^{5/4}}{\sqrt{15N} N^{3/4}}
  \left( \frac{e^{-c(a-1)} \cdot (1 + e^{-c} + 2a(e^{c} - 1) + a^{2} (e^{c} - 1)^{2})}{(e^{c} - 1)^{3}}\right)
\end{align*}
where $a = \left\lfloor \sqrt{\frac{15N}{n}} + 1 \right\rfloor$ and
$c = 4 \pi \sqrt{n/N}$. We increased the exponent on $d$ in the infinite
sum from $3/2$ to $2$ to allow the series to be summed in closed form.

For the terms with $n > 15N$, we use that
\[
  \sum_{d=1}^{\infty} \psi\left(d \sqrt{\frac{n}{N}}\right) \leq
  6 \sqrt{2} \left(\frac{n}{N}\right)^{3/4}
  \sum_{d=1}^{\infty} d^{3/2} e^{-4 \pi d \sqrt{n/N}}.
\]
It is easy to see that $\sum_{d=1}^{\infty} d^{3/2} e^{-4 \pi d \sqrt{n/N}}
\leq 1.000012 e^{-4 \pi \sqrt{n/N}}$, and this gives a corresponding bound on the
infinite sum of values of $\psi$. To bound the other terms in the sum,
we use that $|a_{C^{*}}(n)| \leq C_{Q^{*}}^{{\rm odd}} d(n) \sqrt{n}$ and
that $d(n)^{2} \leq 7.0609n^{3/4}$. Plugging all of this in, the terms
for $n > 15N$ are bounded by
\[
  \frac{60 \cdot 2^{\omega(N)} \left(C_{Q^{*}}^{{\rm odd}}\right)^{2}}{N^{3/4}}
  \sum_{n=15N+1}^{\infty} n^{3/2} e^{-4 \pi \sqrt{n/N}}.
\]
Observe that the sum above is at most
\begin{align*}
  \left(1 + \frac{1}{15N}\right)^{3/2} \int_{15N}^{\infty} x^{3/2}
  e^{-4 \pi \sqrt{x/N}} \, dx &\leq
  2.85 \cdot 10^{-20} \left(1 + \frac{1}{15N}\right)^{3/2} N^{5/2}.
\end{align*}
At the end of this process, we obtain an inequality of the form
\[
  \langle C^{*}, C^{*} \rangle
  \leq C_{1} + C_{2} \left(C_{Q^{*}}^{{\rm odd}}\right)^{2}.
\]
We then have
\[
  C_{Q^{*}}^{{\rm odd}} \leq \sqrt{\frac{u \langle C_{*}, C_{*} \rangle}{B}}
\]
where $B$ is a lower bound on $\langle g_{i}, g_{i} \rangle$. Then we use that
$C_{Q}^{{\rm odd}} = \sqrt{N} C_{Q^{*}}^{{\rm odd}}$ to bound $C_{Q}^{{\rm odd}}$.

We use a similar method to that of Bhargava and Hanke \cite{BH} for computing
a lower bound on the Eisenstein series contribution $a_{E}(n)$, based on part
(b) of Theorem 5.7 of \cite{Hanke}. This requires computing the local densities
$\beta_{p}(n)$, which we do according to the procedure given in \cite{Hanke}.

The end result is an explicit constant $F$
(which we refer to as the $F_{4}$-bound) so that if $n$ is
squarefree and
\[
  F_{4}(n) = \frac{\sqrt{m}}{d(m)} \prod_{\substack{p \nmid N, p | n \\ \chi(p) = -1}} \frac{p-1}{p+1} > F
\]
then $n$ is represented by the form $Q$. We then enumerate all squarefree
integers $n$ for which $F_{4}(n) \leq F$ and check that each of them
is represented by $Q$. To do this, we use a split local cover, a quadratic
form
\[
  R(x,y,z) + dw^{2}
\]
that is represented by $Q$. If $B$ is the largest number satisfying
$F_{4}(n) \leq F$, we compute an approximation of the theta
series of $R$ to precision $C \sqrt{B}$, where $C$ is a constant
(which is chosen to depend on the form $R$). Then, for each squarefree $n$
with $F_{4}(n) \leq F$, we attempt to find an integer $w$ so that
$n - dw^{2}$ is represented by $R$. We choose the parameter $C$ so that
every $n > 5000$ satisfies this, and we manually check that $Q$ represents
every odd number less than $5000$.

\begin{ex}
\label{fundisc}
Form number 10726 is
\[
  Q(x,y,z,w) = x^{2} + 3y^{2} + 3yz + 3yw + 5z^{2} + zw + 34w^{2},
\]
and has discriminant $N = D = 6780$, a fundamental discriminant. The
dimension of $S_{2}(\Gamma_{0}(N), \chi)$ is $1360$. This space has
four Galois-orbits of newforms, of sizes $4$, $4$, $40$, and $1312$.
The explicit method of computing the cusp constant that will be described
in Method 4 below would be impossible for this form.

Proposition~\ref{llow} gives a lower bound
\[
  \langle g_{i}, g_{i} \rangle \geq 0.00001019
\]
for non-CM newforms $g_{i}$. We explicitly compute that there are 48
newforms with CM in $g_{i} \in S_{2}(\Gamma_{0}(N), \chi)$ and the bound above
is valid for them too.  Combining Proposition~\ref{petform} and
Proposition~\ref{cstareq} with the bounds above, we find that
\[
  0.01066 \leq \langle C, C \rangle \leq 0.01079
\]
and from this, we derive that $C_{Q}^{{\rm odd}} \leq 1199.86$. We have that
\[
  a_{E}(n) \geq \frac{28}{151} n \prod_{\substack{p | n, p \nmid N \\ \chi(p) = -1}}
\frac{p-1}{p+1}.
\]
From this, we see that $n$ is represented by $Q$ if $F_{4}(n) \geq 6535$.
The computations run in Magma to derive these bounds for $Q$ took 3 minutes
and 50 seconds.

A separate program (written in C) verifies that any squarefree number $n$
satisfying $F_{4}(n) \leq 6535$ has at most $12$ distinct prime factors,
and is bounded by $8314659320208531$. Of these numbers,
it was necessary to check $4701894614$. This process took
22 minutes and 29 seconds and proves that the form $Q$ represents every
positive odd integer.
\end{ex}

{\bf Method 4}: Explicit computation of the cusp constants.

This method is similar to Method 3, except that we do explicit linear algebra
computations to compute the constant $C_{Q}^{{\rm odd}}$. This method is the approach
Bhargava and Hanke take for all of the cases they consider in \cite{BH}, and
we apply this method to the $6343$ forms $Q$ where none of the first three
methods apply.

The following method is used to compute $C_{Q}^{{\rm odd}}$. If $d$ is a divisor of
$N/\cond(\chi)$, we enumerate representatives of the
Galois orbits of newforms in $S_{2}(\Gamma_{0}(N/d), \chi)$,
say $g_{1}, g_{2}, \ldots, g_{r}$. If the Galois orbit of $g_{i}$ has
size $k_{i}$, we build a basis for
$S_{2}^{\rm new}(\Gamma_{0}(N/d), \chi) \cap \Q[[q]]$ of the form
\[
  \Tr_{K_{i}/\Q}(\alpha^{j} g_{i}) \text{ for } 1 \leq i \leq r, 0 \leq j
  \leq k_{i} - 1,
\]
where $K_{i} = \Q(\alpha)$ is the field generated by adjoining all the
Fourier coefficients of $g_{i}$ to $\Q$. These are then used to build a
basis for the image of
\[
  V(d) : S_{2}(\Gamma_{0}(N/d), \chi) \to S_{2}(\Gamma_{0}(N), \chi).
\]
We do not compute all the coefficients of these forms. Instead
we compute coefficients of the form $dn$ where $\gcd(n,N) = 1$ by computing
the $p$th coefficient of all the forms and using the Hecke relations to
compute the other coefficients. We repeat this process until the matrix
of Fourier expansions has full rank.

Once this basis is built, we solve the linear system (over $\Q$) expressing the
cuspidal part $C$ of $\theta_{Q}$ in terms of the basis. To solve this
system, we work with one value of $d$ at a time, and only use coefficients
of the form $dn$ where $\gcd(n,N) = 1$ to determine the contribution
to $C$ of the image of $V(d) : S_{2}(\Gamma_{0}(N/d), \chi) \to S_{2}(\Gamma_{0}(N), \chi)$. Once we have the representation of $C$ in terms of
the full basis for $S_{2}(\Gamma_{0}(N), \chi)$, we numerically approximate the
embeddings of the $\alpha^{j}$ and use these to compute $C_{Q}^{{\rm odd}}$.

\begin{ex}
Form 22145 is
\[
  Q(x,y,z,w) = x^{2} - xz + 2y^{2} + yz - 2yw + 5z^{2} + zw + 29w^{2}.
\]
For this $Q$, $\theta_{Q} \in M_{2}(\Gamma_{0}(4200), \chi_{168})$.
The dimension of $S_{2}(\Gamma_{0}(4200), \chi_{168})$ is $936$. There
are 19 Galois conjugacy classes of newforms of levels $168$, $840$,
and $4200$, the largest of which has size $160$.

The $d = 1$ space has dimension $752$, and we need to compute the
$p$th coefficient of all newforms of level dividing $4200$ for $p \leq 197$.
Once these are computed, it is straightforward to find bases for the
$d = 5$ and $d = 25$ spaces (of dimensions
$156$ and $28$, respectively). Solving the linear system gives that
$C_{Q}^{{\rm odd}} \approx 31.0537$.
For odd squarefree $n$, we have
\[
  a_{E}(n) \geq \frac{28}{117} n \prod_{\substack{p | n, p \nmid N \\ \chi(p) = -1}}
\frac{p-1}{p+1}.
\]
This shows that if $n$ is a squarefree odd integer and
$F_{4}(n) > 131.0575$, then $n$ is represented by $Q$. The bound on $F_{4}$
is quite small, and it is only necessary to test $638080$ integers. However,
computing the bound on $F_{4}$ required almost a day of computation, due to
the difficulty of computing the constant $C_{Q}^{{\rm odd}}$. The result is that $Q$
represents all positive odd integers.
\end{ex}

\begin{proof}[Proof of the 451-Theorem]
  Assume Conjecture~\ref{Kaplansky}.
  The computations show that every one of the $24888$ forms considered
  locally represents all positive odd integers, and in each case we
  are able to determine precisely the list of squarefree odd
  exceptions for each form.  Moreover, every odd universal lattice
  contains one of the 23 odd universal ternary escalators, or one of the
  $24888$. Of the $24888$, there are $23519$ that represent all
  positive odds, and $1359$ that have exceptions. Of these $1359$,
  there are $15$ forms that have an exception which is not a critical
  integer. (These
  are forms 1044, 8988, 9011, 9016, 11761, 16366, 16372
  17798, 24290, 24311, 24328, 24435, 24463, 24504, and 24817.)
  It is necessary to check that each escalation of these forms
  represents all non-critical positive odds. The most time-consuming form to
deal with is form $16366$,
\[
  Q(x,y,z,w) = x^{2} + xy + 3y^{2} + 4z^{2} + 77w^{2}
\]
which has truant $143$, and fails to represent $187$, $231$, $385$, $451$,
$627$, $935$, $1111$, $1419$, $1903$, and $2387$. We compute all escalations
of it (which requires consideration of more than 10 million Gram matrices),
and find among its escalations forms that have truants $187$, $231$, $385$,
and $451$, but not $627$, $935$, $1111$, $1419$, $1903$ or $2387$. This
concludes the proof that every positive-definite quadratic form representing
the $46$ critical integers represents all positive odd integers.
\end{proof}

\begin{rem}
The program and log files used to prove the 451-Theorem are available at\\
{\tt http://www.wfu.edu/$\sim$rouseja/451}.
\end{rem}

We will now show that each critical integer is necessary.

\begin{proof}[Proof of Corollary~\ref{crit}]
Each of the critical integers occurs
as the truant for some form $Q$ (see Appendix~\ref{truant_table}).
Using the same trick as in \cite{BH}, if $Q(\vec{x})$ is any
form with truant $t$, consider the form
\[
  Q' = Q(\vec{x}) + (t+1) y^{2} + (t+1) z^{2} + (t+1) w^{2} + (t+1)v^{2} + (2t+1)u^{2}.
\]
This form fails to represent $t$. However, since every positive
integer is expressible as a sum of four squares, if $Q$
represents the odd number $a$, then every number $\equiv a \pmod{t+1}$
is represented by $Q'$. This accounts for all odd numbers
except those $\equiv t \pmod{t+1}$. Taking $Q = 0$ and $u = 1$,
we see that $Q'$ represents all numbers $\equiv t \pmod{t+1}$
that are greater than or equal to $2t + 1$. Hence,
$t$ is the unique positive odd integer which is not represented by $Q'$.
\end{proof}

As an application of the 451-Theorem, we will prove Corollary~\ref{smallbig}.

\begin{proof}[Proof of Corollary~\ref{smallbig}]
If $Q$ is a quadratic form with corresponding lattice $L$ that
represents every positive odd integer less than $451$, then $L$
contains as a sublattice one of the $24888$ we considered
above. Of these, only forms $1048$, $16327$, $16334$, $16336$
and $16366$ have $451$ as an exception. Each of these has a nicely
embedded regular ternary, and the application of method 2 shows that
each of these represents all odd positive integers $n$ that are
not multiples of $11^{2}$, with a finite and explicit set of exceptions.
For forms $1048$, $16334$, $16336$ and $16366$ it is easy to see that
all multiples of $11^{2}$ are represented.

Each of $1048$, $16334$ and $16336$ have truant $143$ and no exceptions larger than $451$.
For form $16366$, we computed all escalations in the course of proving the 451-Theorem and
found that none of them have squarefree exceptions greater than $451$.

However, form $16327$
\[
  Q(x,y,z,w) = x^{2} + xy + 3y^{2} + 4z^{2} + 66w^{2}
\]
is anisotropic at $11$. The form $Q$
represents all squarefree odd integers that are not multiples of $11^{2}$ except $319$ and
$451$. A computer calculation shows that $r_{Q}(121n) = r_{Q}(n)$
for all positive integers $n$ and hence, the odd integers not represented
by $Q$ are those of the form $319 \cdot 11^{2k}$ and $451 \cdot 11^{2k}$.
It is therefore necessary to compute all escalations of $Q$ by
$319$, find those that fail to represent $451$, and check
that each of these represents $451 \cdot 11^{2} = 54571$. We find $21$
five-dimensional escalations that fail to represent
$451$ and each of these represents $54571$.
\end{proof}

As an application of the 451-Theorem we will classify those
quaternary forms that represent all odd positive integers.

\begin{proof}[Proof of Corollary~\ref{oddclass}]
The successive minima of quaternary escalator lattices are
bounded by $1$, $3$, $7$, and $77$. We enumerate all
Minkowski-reduced lattices with successive minima less than or
equal to these, apply the 451-Theorem to determine which
represent all positive odds, and determine those that represent
one of the odd universal ternary forms. A list of the $21756$
forms that were found is available on the website mentioned
above.
\end{proof}

\section{Conditional proof of Conjecture~\ref{Kaplansky}}
\label{kappf}

We begin by recalling the theory of modular forms of half-integer
weight.  If $\lambda$ is a positive integer, let $S_{\lambda +
  \frac{1}{2}}(\Gamma_{0}(4N), \chi)$ denote the vector space of cusp
forms of weight $\lambda + \frac{1}{2}$ on $\Gamma_{0}(4N)$ with
character $\chi$. We denote by by $T(p^{2})$
the usual index $p^{2}$ Hecke operator on $S_{\lambda +
  1/2}(\Gamma_{0}(4N), \chi)$. Next, we recall the Shimura lifting.

\begin{thmnonum}[\cite{Shi}]
Suppose that $f(z) = \sum_{n=1}^{\infty} a(n) q^{n} \in
S_{\lambda + 1/2}(\Gamma_{0}(4N), \chi)$. For each squarefree integer $t$,
let
\[
  \mathcal{S}_{t}(f(z)) = \sum_{n=1}^{\infty}
  \left(\sum_{d | n} \chi(d) \legen{(-1)^{\lambda} t}{d}
  d^{\lambda - 1} a(t(n/d)^{2})\right) q^{n}.
\]
Then, $\mathcal{S}_{t}(f(z)) \in M_{2\lambda}(\Gamma_{0}(2N), \chi^{2})$. It is
a cusp form if $\lambda > 1$ and if $\lambda = 1$ it is a cusp form if
$f(z)$ is orthogonal to all cusp forms $\sum_{n=1}^{\infty} \psi(n) n q^{n^{2}}$
where $\psi$ is an odd Dirichlet character.
\end{thmnonum}

One can show using the definition that if $p$ is a prime and
$p \nmid 4tN$, then $\mathcal{S}_{t}(f | T(p^{2})) = \mathcal{S}_{t}(f) | T(p)$.
In \cite{Wald}, Waldspurger relates the Fourier coefficients of
a half-integer weight Hecke eigenform $f$ with the central critical $L$-values
of the twists of the integer weight newform $F$
with the same Hecke eigenvalues. If we have a newform
$F(z) = \sum_{n=1}^{\infty} b(n) q^{n} \in S_{2}^{{\rm new}}(\Gamma_{0}(N))$,
and $\chi$ is a quadratic Dirichlet character, we define
$F \otimes \chi$ to be the unique newform whose $n$th Fourier coefficient
is $b(n) \chi(n)$ if $\gcd(n, N \cdot \cond(\chi)) = 1$.

\begin{thmnonum}[\cite{Wald}, Corollaire 2, p. 379]
Suppose that $f \in S_{\lambda + 1/2}(\Gamma_{0}(N), \chi)$ is a half-integer
weight modular form and $f | T(p^{2}) = \lambda(p) f$ for all $p \nmid N$
with Fourier expansion $f(z) = \sum_{n=1}^{\infty} a(n) q^{n}$. If
$F(z) \in S_{2 \lambda}(\Gamma_{0}(N), \chi^{2})$ is an integer weight newform
with $F(z) | T(p) = \lambda(p) g$ for all $p \nmid N$ and $n_{1}$ and
$n_{2}$ are two squarefree positive integers with
$n_{1}/n_{2} \in \left(\Q_{p}^{\times}\right)^{2}$ for all $p | N$, then
\[
  a(n_{1})^{2} L(F \otimes \chi^{-1} \chi_{n_{2} (-1)^{\lambda}}, 1/2)
  \chi(n_{2}/n_{1}) n_{2}^{\lambda - 1/2}
  = a(n_{2})^{2} L(F \otimes \chi^{-1} \chi_{n_{1} (-1)^{\lambda}}, 1/2)
  n_{1}^{\lambda - 1/2}.
\]
\end{thmnonum}

If $Q$ is a positive-definite, integer-valued ternary quadratic form,
then $\theta_{Q}(z) = \sum_{n=0}^{\infty} r_{Q}(n) q^{n} \in
  M_{3/2}(\Gamma_{0}(4N), \chi)$.
We may then decompose $\theta_{Q}(z) = E(z) + C(z)$
where $E(z)$ is a half-integer weight Cohen-Eisenstein series,
and $C(z)$ is a cusp form. We have $E(z) = \sum_{n=0}^{\infty} a_{E}(n) q^{n}$,
where if $n \geq 1$ is squarefree, then
\[
  a_{E}(n) = \frac{24 h(-nM)}{M w(-nM)} \prod_{p | 2N}
  \beta_{p}(n) \cdot \frac{1 - (1/p) \chi(p) \legen{n}{p}}{1 - 1/p^{2}}.
\]
Here $M$ is a rational number which depends on $n \pmod{8N^{2}}$
with the property that $nM$ is a fundamental discriminant. Here $h(-nM)$
is the class number of the ring of integers in $\Q(\sqrt{-nM})$
and $w(-nM)$ is half the number of roots of unity in $\Q(\sqrt{-nM})$. From
Siegel's work, we have the ineffective lower bound
$h(-D) \gg D^{1/2 - \epsilon}$, but the strongest effective lower bound we have
is due to the work of Goldfeld \cite{Goldfeld}, Gross and Zagier \cite{GZ} and
has the form $h(-D) \gg \log(D)^{1 - \epsilon}$. For this reason, there is
no general method to determine unconditionally the integers represented by a
positive-definite ternary quadratic form.

We may decompose the cusp form contribution as a linear combination
of half-integer weight Hecke eigenforms $C(z) = \sum_{i} c_{i} f_{i}(z)$.
Each $f_{i}(z)$ either has the form $\sum \psi(n) n q^{dn^{2}}$, in which
case its nonzero Fourier coefficients are supported on a single square-class,
or Waldspurger's theorem applies, and gives that if
\[
  f_{i}(z) = \sum_{n=1}^{\infty} b(n) q^{n},
\]
then
\[
  |b(n)| = d n^{1/4} |L(F_{i} \otimes \chi_{bn}, 1/2)|
\]
for some constants $b$ and $d$ which depend on the $\Q_{p}$-square classes
of $n$ (provided we can find a value of $n$ in the $\Q_{p}$-square classes
so that the coefficient of $f_{i}$ and the central $L$-value of the
corresponding twist of $F_{i}$ are nonzero). The best currently known
subconvexity estimate for $|L(F_{i} \otimes \chi_{bn}, 1/2)|$ is due to Blomer and
Harcos (\cite{BlomerHarcos}, Corollary 2) and gives that
\[
  |b(n)| \ll n^{7/16 + \epsilon}.
\]
However, the Generalized Riemann Hypothesis implies that
$|b(n)| \ll n^{1/4 + \epsilon}$. In \cite{OnoSound}, Ono and Soundararajan
pioneered a method to conditionally determine the integers represented
by a ternary quadratic form and used it to prove that Ramanujan's form
$x^{2} + y^{2} + 10z^{2}$ represents every odd number greater than $2719$.
This method was generalized by Kane \cite{Kane} and refined by Chandee
\cite{Chandee}. We prove Conjecture~\ref{Kaplansky} by using Theorem
2.1 and Proposition 4.1 of \cite{Chandee} (which assume the
Generalized Riemann Hypothesis) to bound
$|L(F_{i} \otimes \chi_{bn}, 1/2)|$ and
\[
  L(1, \chi_{nM}) = \frac{\pi h(-nM)}{\sqrt{nM} w(-nM)}.
\]

\begin{proof}[Proof of Theorem~\ref{GRHimplies}]
For $Q = x^{2} + 2y^{2} + 5z^{2} + xz$, we have
\begin{align*}
  \theta_{Q}(z) &= 1 + 2q + 2q^{2} + 4q^{3} + 2q^{4} + 4q^{5} + \cdots\\
  &= \sum_{n=0}^{\infty} r_{Q}(n) q^{n} \in M_{3/2}(\Gamma_{0}(152), \chi_{152}).
\end{align*}
The genus of $Q$ has size $2$, and the other form is $R = x^{2} + y^{2}
+ 13z^{2} - xy - xz + yz$. We have
\begin{align*}
  E &= \frac{3}{5} \theta_{Q} + \frac{2}{5} \theta_{R}
  = 1 + \frac{18}{5} q + \frac{6}{5} q^{2} + \frac{24}{5} q^{3} + \frac{18}{5}
  q^{4} + \frac{12}{5} q^{5} + \cdots\\
  C &= \theta_{Q} - E = -\frac{8}{5} q + \frac{4}{5} q^{2} - \frac{4}{5}
  q^{3} - \frac{8}{5} q^{4} + \frac{8}{5} q^{5} + \cdots.
\end{align*}
The Shimura lift $\mathcal{S}_{3} : S_{3/2}(\Gamma_{0}(152), \chi_{152}) \to
M_{2}(\Gamma_{0}(76))$ is injective, and $\mathcal{S}_{3}(C)$ is a constant
times the newform
\[
  F_{1}(z) = q + q^{2} - q^{3} + q^{4} - 4q^{5} - q^{6} + \cdots
  \in S_{2}(\Gamma_{0}(38)),
\]
which corresponds to the elliptic curve
\[
  E_{1} : y^{2} + xy + y = x^{3} + x^{2} + 1.
\]
For each pair $(n_{1}, n_{2}) \in (\Q_{2}^{\times} /
\left(\Q_{2}^{\times}\right)^{2}) \times (\Q_{19}^{\times} /
\left(\Q_{19}^{\times}\right)^{2})$ with $\ord_{2}(n_{1}) = 0$, we
compute constants $a$, $b$, and $d$ so that if
$n$ is a squarefree integer with $n/n_{1} \in \left(\Q_{2}^{\times}\right)^{2}$
and $n/n_{2} \in \left(\Q_{19}^{\times}\right)^{2}$, we have
\[
  r_{Q}(n) = a h(-bn) \pm d n^{1/4} \sqrt{L(F_{1} \otimes \chi_{-152n}, 1/2)}.
\]
For $n_{1} = n_{2} = 1$, we have $a = 3/5$, $b = 152$,
and $d \approx 0.9150328989$. This shows that if $r_{Q}(n) = 0$,
then
\[
  \frac{\sqrt{L(F_{1} \otimes \chi_{-152n},1/2)}}{L(1,\chi_{-152n})}
  \geq 2.573276 n^{1/4}.
\]
On the other hand, computations using Chandee's theorems give that
\[
  \frac{\sqrt{L(F_{1} \otimes \chi_{-152n},1/2)}}{L(1,\chi_{-152n})}
  \leq 13.848476 \cdot n^{0.1239756}.
\]
Comparing these two results, we see that if $n$ is a 2-adic and a 19-adic
square, then $r_{Q}(n) > 0$ if $n \geq 630654$, assuming the
Generalized Riemann Hypothesis. We obtain the
same bounds on the other squareclasses $(n_{1},n_{2})$ where
$\ord_{19}(n_{2}) = 0$. On the squareclasses where $n_{2} = 19$,
we obtain smaller bounds. Finally, it is possible to prove that
$C$ vanishes identically on squareclasses where $n_{2} = 38$. To check that
every odd number less than this bound is represented, we compute
the theta series of $S = x^{2} + xz + 5z^{2}$ up to $q^{630654}$. For each odd
number $n \leq 630654$, we check if $n - 2y^{2}$ is represented by
$S$ for some $y \leq \sqrt{n/2}$. This computation takes 2.79 seconds.

For $Q = x^{2} + 3y^{2} + 6z^{2} + xy + 2yz$, the genus again has size $2$,
and $\theta_{Q} \in M_{3/2}(\Gamma_{0}(248), \chi_{248})$. We find
that $\mathcal{S}_{1} : S_{3/2}(\Gamma_{0}(248), \chi_{248}) \to M_{2}(\Gamma_{0}(124))$ is injective. If $C$ is the cuspidal part of $\theta_{Q}$, then
$\mathcal{S}_{1}(C)$ is some constant times the newform
$F_{2} \in S_{2}(\Gamma_{0}(62))$ that corresponds to the elliptic curve
\[
  E_{2} : y^{2} + xy + y = x^{3} - x^{2} - x + 1
\]
with conductor $62$. Again for each pair $(n_{1},n_{2}) \in \Q_{2}^{\times} / (\Q_{2}^{\times})^{2} \times \Q_{31}^{\times}/(\Q_{31}^{\times})^{2}$, we find $a$, $b$,
and $d$ so that
\[
  r_{Q}(n) = a h(-bn) \pm d n^{1/4} \sqrt{L(F_{2} \otimes \chi_{-248n}, 1/2)}.
\]
For $(n_{1},n_{2}) = 1$, we have $a = 3/8$, $b = 248$, and
$d \approx 0.6630028204$. If $r_{Q}(n) = 0$, we get that
\[
  \frac{\sqrt{L(F_{2} \otimes \chi_{-248n},1/2)}}{L(1,\chi_{-248n})}
  \geq 2.835253 n^{1/4}
\]
and using Chandee's theorems we get
\[
  \frac{\sqrt{L(F_{2} \otimes \chi_{-248n},1/2)}}{L(1,\chi_{-248n})}
  \leq 14.492987 \cdot n^{0.1239756}.
\]
This proves that if $n \geq 419230$ and $n$ is a 2-adic and 31-adic square,
then $r_{Q}(n) > 0$ (assuming GRH). We obtain equal or smaller bounds
on the other square classes. To check up to this bound, we use that
\[
  4Q = (2x+y)^{2} + 11y^{2} + 8yz + 24z^{2}.
\]
If $4n$ is represented by $w^{2} + 11y^{2} + 8yz + 24z^{2}$,
then $w \equiv y \pmod{2}$ and hence if we set $x = \frac{w-y}{2}$, we get that
$n = x^{2} + 3y^{2} + 6z^{2} + xy + 2yz$. Hence $n$ is represented
by $Q$ if and only if $4n$ is represented by $w^{2} + 11y^{2} + 8yz + 24z^{2}$.
We compute the theta series for $S = 11y^{2} + 8yz + 24z^{2}$ up to
$q^{1680000}$. Then, for each number $m \equiv 4 \pmod{8}$ between
$4$ and $1680000$, we check that $m-w^{2}$ is represented by $S$ for some
$w$. We find that this is true, and the computation takes 2.53 seconds.

Finally, for the form $Q = x^{2} + 3y^{2} + 7z^{2} + xy + xz$, we have
$\theta_{Q} \in M_{3/2}(\Gamma_{0}(296), \chi_{296})$. We use the maps
$\mathcal{S}_{1} : S_{3/2}(\Gamma_{0}(296), \chi_{296}) \to M_{2}(\Gamma_{0}(148))$
and $\mathcal{S}_{5} : S_{3/2}(\Gamma_{0}(296), \chi_{296}) \to M_{2}(\Gamma_{0}(148))$. We find that neither are injective, but that the intersection of their
kernels is zero. If $C$ is the cuspidal part of $\theta_{Q}$, then $C$ is
a linear combination of two eigenforms whose Shimura lifts are the two newforms
\begin{align*}
  F_{3}^{+} &= q + q^{2} + \frac{-1 + \sqrt{5}}{2} q^{3} + q^{4} +
  \frac{1 - 3 \sqrt{5}}{2} q^{5} + \cdots\\
  F_{3}^{-} &= q + q^{2} + \frac{-1 - \sqrt{5}}{2} q^{3} + q^{4} +
  \frac{1 + 3 \sqrt{5}}{2} q^{5} + \cdots
\end{align*}
of level $74$. For each pair $(n_{1},n_{2}) \in \Q_{2}^{\times} / (\Q_{2}^{\times})^{2} \times \Q_{37}^{\times}/(\Q_{37}^{\times})^{2}$, we find constants
$a$, $b$, $d_{1}$ and $d_{2}$ so that
\[
  r_{Q}(n) = a h(-bn) \pm d_{1} n^{1/4} \sqrt{L(F_{3}^{+} \otimes \chi_{-296n},
1/2)} \pm d_{2} n^{1/4} \sqrt{L(F_{3}^{-} \otimes \chi_{-296n}, 1/2)}.
\]
For $n_{1} = n_{2} = 1$, we have $a = 6/19$, $b = -296$, $d_{1} \approx 0.2092923830$ and $d_{2} \approx 0.5342698872$. Hence, if $r_{Q}(n) = 0$, we have
\[
  d_{1} \frac{\sqrt{L(F_{3}^{+} \otimes \chi_{-296n},1/2)}}{L(1,\chi_{-296n})}
  + d_{2} \frac{\sqrt{L(F_{3}^{-} \otimes \chi_{-296n},1/2)}}{L(1,\chi_{-296n})}
  \geq 1.729392 n^{1/4}.
\]
Applying Chandee's theorems, we get
\begin{align*}
  & \frac{\sqrt{L(F_{3}^{+} \otimes \chi_{-296n},1/2)}}{L(1,\chi_{-296n})}
  \leq 13.678621 n^{0.1239756}, \text{ and }\\
  & \frac{\sqrt{L(F_{3}^{-} \otimes \chi_{-296n},1/2)}}{L(1,\chi_{-296n})}
  \leq 15.592398 n^{0.1239756}.
\end{align*}
It follows that if $r_{Q}(n) = 0$, then $n \leq 2727720$, assuming GRH.
We find equal or smaller bounds on the other square classes. To check up to
this bound, we use that
\[
  4Q = (2x + y + z)^{2} + 11y^{2} - 2yz + 27z^{2}.
\]
If $4n$ is represented by $w^{2} + 11y^{2} - 2yz + 27z^{2}$, then
$w \equiv 11y^{2} - 2yz + 27z^{2} \pmod{2}$, which implies that $w \equiv
y+z \pmod{2}$. Setting $x = \frac{w - (y+z)}{2}$, we obtain
$n = x^{2} + 3y^{2} + 7z^{2} + xy + xz$. Thus, $n$ is represented by $Q$ if and
only if $4n$ is represented by $w^{2} + 11y^{2} - 2yz + 27z^{2}$. We compute
the theta series of $S = 11y^{2} - 2yz + 27z^{2}$ up to $q^{10912000}$, and check
that for every number $m \equiv 4 \pmod{8}$ less than $10912000$, $m - w^{2}$
is represented by $S$ for some integer $w$. We find that this is
true, and the computation takes 16.76 seconds.

This completes the proof of Theorem~\ref{GRHimplies},
assuming the Generalized Riemann Hypothesis.
\end{proof}

\appendix
\section{Table of quadratic forms with given truants}
\label{truant_table}

\small
\begin{center}
\begin{tabular}{l|c}
Form & Truant\\
\hline
$\emptyset$ & $1$\\
$x^{2}$ & $3$\\
$x^{2} + 2y^{2}$ & $5$\\
$x^{2} + 3y^{2} + xy$ & $7$\\
$x^{2} + 3y^{2} + 4z^{2} + yz$ & $11$\\
$x^{2} + 3y^{2} + 6z^{2} + xy + yz$ & $13$\\
$x^{2} + y^{2} + 3z^{2}$ & $15$\\
$x^{2} + 2y^{2} + 3z^{2} + xy + xz + 2yz$ & $17$\\
$x^{2} + 3y^{2} + 7z^{2} + xy + yz$ & $19$\\
$x^{2} + 3y^{2} + 3z^{2} + xy + xz + 2yz$ & $21$\\
$x^{2} + 2y^{2} + 3z^{2} + yz$ & $23$\\
$x^{2} + 3y^{2} + 3z^{2} + xy$ & $29$\\
$x^{2} + 2y^{2} + 4z^{2} + yz$ & $31$\\
$x^{2} + 3y^{2} + 4z^{2} + 10w^{2} + 2yw$ & $33$\\
$x^{2} + 3y^{2} + 5z^{2} + 3yz$ & $35$\\
$x^{2} + 2y^{2} + 5z^{2} + 12w^{2} + xz + xw + yz + 3zw$ & $37$\\
$x^{2} + 3y^{2} + 5z^{2} + 13w^{2} + xy + xz + yz$ & $39$\\
$x^{2} + 3y^{2} + 4z^{2} + xz + 2yz$ & $41$\\
$x^{2} + 3y^{2} + 5z^{2} + 15w^{2} + xw + yz + 2yw + 2zw$ & $47$\\
$x^{2} + 3y^{2} + 4z^{2} + 15w^{2} + xw + 3yz + zw$ & $51$\\
$x^{2} + 3y^{2} + 5z^{2} + 21w^{2} + xz + 2yz + yw + 4zw$ & $53$\\
$x^{2} + 3y^{2} + 7z^{2} + 9w^{2} + xy + 2yw$ & $57$\\
$x^{2} + 3y^{2} + 5z^{2} + 16w^{2} + yz + 2yw + 3zw$ & $59$\\
$x^{2} + y^{2} + 3z^{2} + yz$ & $77$\\
$x^{2} + 3y^{2} + 5z^{2} + 21w^{2} + 3yz$ & $83$\\
$x^{2} + 3y^{2} + 5z^{2} + 23w^{2} + xw + 3yz + 4zw$ & $85$\\
$x^{2} + 3y^{2} + 5z^{2} + 9w^{2} + xz + 3yw$ & $87$\\
$x^{2} + 3y^{2} + 5z^{2} + 27w^{2} + xz + 2yz + yw + 2zw$ & $89$\\
$x^{2} + 3y^{2} + 7z^{2} + 9w^{2} + 21v^{2} + xy + yw + 7zv$ & $91$\\
$x^{2} + 2y^{2} + 4z^{2} + 28w^{2} + yz + zw$ & $93$\\
$x^{2} + 3y^{2} + 4z^{2} + 11w^{2} + xw + 2zw$ & $105$\\
$x^{2} + 3y^{2} + 5z^{2} + 31w^{2} + 3yz + 3yw$ & $119$\\
$x^{2} + 3y^{2} + 4z^{2} + 9w^{2} + 3yw$ & $123$\\
$x^{2} + 3y^{2} + 7z^{2} + 19w^{2} + 57v^{2} + xy + yz + 17wv$ & $133$\\
$x^{2} + 3y^{2} + 5z^{2} + 26w^{2} + xw + 3yz + 3zw$ & $137$\\
$x^{2} + y^{2} + 3z^{2} + 47w^{2} + xw + yz$ & $143$\\
$x^{2} + 3y^{2} + 3z^{2} + 29w^{2} + xy + 2yz$ & $145$\\
$x^{2} + 3y^{2} + 3z^{2} + 20w^{2} + xy + 3zw$ & $187$\\
$x^{2} + 3y^{2} + 6z^{2} + 13w^{2} + xy + yz$ & $195$\\
$x^{2} + 2y^{2} + 4z^{2} + 29w^{2} + 58v^{2} + xz + yz$ & $203$\\
$x^{2} + 3y^{2} + 4z^{2} + 41w^{2} + xz + 2yz$ & $205$\\
$x^{2} + y^{2} + 3z^{2} + 36w^{2} + xw + yz$ & $209$\\
$x^{2} + 3y^{2} + 4z^{2} + 77w^{2} + 143v^{2} + xy + 15wv$ & $231$\\
$x^{2} + 3y^{2} + 4z^{2} + 33w^{2} + xy$ & $319$\\
$x^{2} + 3y^{2} + 4z^{2} + 77w^{2} + 143v^{2} + xy + 22wv$ & $385$\\
$x^{2} + 3y^{2} + 4z^{2} + 77w^{2} + 143v^{2} + xy + 33wv$ & $451$\\
\end{tabular}
\end{center}
\normalsize

\bibliographystyle{amsplain}
\bibliography{451refs}

\end{document}